\documentclass[11pt, a4paper]{amsart}
\usepackage{amsmath, amsthm, amscd, amssymb, amsfonts, amsxtra, amssymb, latexsym}
\usepackage{verbatim}
\usepackage{graphicx} 
\usepackage{enumerate}
\usepackage[T1]{fontenc}
\usepackage{lmodern}
\usepackage{float}

\usepackage{hyperref}
\hypersetup{
	colorlinks = true,
	allcolors  = blue
}

\usepackage{xcolor}

\newcommand{\sk}{\smallskip}
\newcommand{\msk}{\medskip}

\newcommand{\N}{\mathbb{N}}
\newcommand{\Z}{\mathbb{Z}}
\newcommand{\Q}{\mathbb{Q}}
\newcommand{\R}{\mathbb{R}}
\newcommand{\C}{\mathbb{C}}

\newcommand{\Tr}{\mathrm{Tr}}
\newcommand{\G}{\Gamma}

\numberwithin{equation}{section}

\newtheorem{thm}{Theorem}[section]
\newtheorem{prop}[thm]{Proposition}
\newtheorem{lem}[thm]{Lemma}
\newtheorem{coro}[thm]{Corollary}

\theoremstyle{definition}
\newtheorem{rem}[thm]{Remark}
\newtheorem{exam}[thm]{Example}
\newtheorem{defi}[thm]{Definition}

\begin{document} \sloppy
\title[The energy of a family of mirror di-Cayley (sum) graphs]{The energy of a family of mirror di-Cayley (sum) graphs: equienergy and moments}

\author{Paula M.\@ Chiapparoli, Ricardo A.\@ Podest\'a}

\dedicatory{\today}

\thanks{Partially supported by CONICET, FonCyT and SECyT-UNC}
\address{R.A.\@ Podest\'a. FaMAF--CIEM (CONICET), Universidad Nacional de C\'ordoba. 
	\newline Av.\@ Medina Allende 2144, Ciudad Universitaria, (5000), C\'ordoba, Argentina. \newline
	{\it E-mail: ricardo.podesta@unc.edu.ar}}
\address{P.M.\@ Chiapparoli. FaMAF--CIEM (CONICET), Universidad Nacional de C\'ordoba. 
	\newline	Av.\@ Medina Allende 2144, Ciudad Universitaria, (5000), C\'ordoba, Argentina.
	\newline {\it E-mail: paula.chiapparoli@mi.unc.edu.ar}}

\begin{abstract}
For a group $G$ and subsets $S,T \subset G$ we consider a family of mirror di-Cayley graphs $MX(G;S,T)$ and mirror di-Cayley sums graphs $MX^+(G;S,T)$, namely those with $T=\{e\}, S$ or $S \cup \{e\}$. We refer to them indistinctly by $MX^*(G;S,T)$. 
We can think of $MX^*(G;S,T)$ as two copies of the Cayley (sum) graph $X^*(G,S)$ joined by edges determined by the connection set $T$.
Recently, in the work \textit{Isospectral Cayley graphs with even and odd spectrum} (\cite{ChP}), we study the spectrum of these graphs and several isospectrality problems. Here, we compute the energy and spectral $k$-moments of $MX^*(G;S,T)$ in terms of those of the underlying Cayley graphs $X(G,S)$. Then, we study energetic problems like hypo-, order-, and hyper-energeticity of these graphs. Finally, we give conditions for the existence of equienergetic pairs of non-isomorphic MDCGs.
\end{abstract}

\maketitle

\section{Introduction} \label{sec: intro} 
This work deals with the energetic and equienergetic properties of the family $\mathcal{F}$ of mirror di-Cayley (sum) graphs recently defined in \cite{ChP}.

Let $G$ be a group and $S_\ell, S_r, T \subset G$. 
Recently, in \cite{ChP5}, we have introduced and studied di-Cayley graphs in general. In \cite{ChP4},  
we focus deeper on the family of mirror di-Cayley (sum) graphs, that is those di-Cayley (sum) graphs $DX^*(G;S_\ell, S_r, T)$ having equal connection sets $S_ \ell=S_r=S$, which in this case are denoted by 
$MX^*(G;S,T)$. 
More precisely, the \textit{mirror di-Cayley graph} (MDCG for short) 
$$MX^*(G;S,T)$$ 
is the graph having vertex set $G \times \Z_2$ and where the edges are the edges of the Cayley graph $X^*(G,S)$ replicated in $X^*(G,S) \times \{0\}$ and $X^*(G,S) \times \{1\}$ (called the mirrors) and the crossed edges, that is the edges between the mirrors (see Definition~\ref{def: dicayleys} below). 
Here, and in all the paper, $X^*(G,S)$ means the Cayley graph $X(G,S)$ and the Cayley sum graph $X^+(G,S)$ considered together 
(and similarly for the mirror di-Cayley (sum) graphs $MX^*(G;S,T)$).
We refer to \cite{ChP4} for basic properties of mirror di-Cayley (sum) graphs and their spectrum in general.

Here, we will restrict ourselves to an interesting family of MDCG's previously considered in \cite{ChP}, that we called $\mathcal{F}$;
namely, those having 
	$T \in \{\{e\}, S, S\cup \{e\}\}$. 
That is, we will consider the graphs 
	$$ MX^*(G;S,\{e\}), \qquad MX^*(G;S,S) \qquad \text{and} \qquad MX^*(G;S,S\cup \{e\}).$$ 
Other choices of $T$ are also of interest. However, we restrict ourselves to these subsets since the associated graphs have nice decompositions as products of graphs (see Section 2). 

In \cite{ChP}, we focus on the computation of the spectrum and the study of isospectrality in $\mathcal{F}$. On the other hand, here, we will study the energy
and equienergeticity for the graphs in the family $\mathcal{F}$. 
In this way, the present work is a natural continuation of \cite{ChP} which complements the spectral study of MDCGs in the family $\mathcal{F}$.

\subsection*{Graph spectral definitions}
Let $\G=(V,E)$ be a graph with $n$ vertices, where $V$ denotes the vertex set and $E$ the edge set. The eigenvalues of $\Gamma$ are the eigenvalues $\{\lambda_i\}_{i=1}^{n}$ of its adjacency matrix $A$. The \textit{spectrum} of $\Gamma$, denoted
	$$ Spec(\Gamma)=\{[\lambda_{i_1}]^{m_{i_1}},\cdots,[\lambda_{i_s}]^{m_{i_s}}\} $$
is the (multi)set of all the different eigenvalues $\{\lambda_{i_j}\}$ of $\Gamma$, counted with their multiplicities $\{m_{i_j}\}$, that is $m(\lambda_{i_j})=m_{i_j}$. The spectrum of $\Gamma$ is called \textit{symmetric} if $Spec(\G)=-Spec(\G)$, that is the multiplicities of $\lambda$ and $-\lambda$ coincide for any $\lambda$, in symbols 
	$m(\lambda) = m(-\lambda)$ 
for every $\lambda\in Spec(\Gamma)$. The spectrum of $\G$ is said to be \textit{integral} or \textit{real} if $Spec(\Gamma) \subset\mathbb{Z}$ or $Spec(\Gamma) \subset \mathbb{R} \smallsetminus \Z$, respectively. 

We recall that if $\Gamma$ is a $k$-regular graph, then $\lambda_1 = k$. Moreover, the multiplicity of $\lambda_1$ equals the number of connected components of $G$ and hence $\Gamma$ is connected if and only if $m(\lambda_1) = 1$. Spectrally, a $k$-regular graph $\G$ is bipartite if and only if the spectrum of $\G$ is symmetric, which happens if and only if $-k$ is an eigenvalue of $\Gamma$.

We now recall some nice spectral invariants of a graph $\G$.
The \textit{energy} of $\Gamma$ is defined by
\begin{equation} \label{eq: energy}
	\mathcal{E}(\Gamma) = \sum_{\lambda \in Spec(\G)} |\lambda|.
\end{equation}
Thus, $\mathcal{E}(\G)\ge 0$. 
For any $\ell \in \N_0$, the \textit{$\ell$-th spectral moment} $M_\ell(\G)$ of a graph $\G$ is
\begin{equation} \label{eq: moments}
	M_\ell(\G) = \sum_{\lambda \in Spec(\G)} \lambda^\ell. 
\end{equation}
Note that $M_\ell(\G) \in \C$ but we have that $M_0(\G)=|G|$ and $M_1(\G)=0$. 

Similarly, we now introduce the energetic moments of a graph.
\begin{defi} \label{defi: energetic moments}
Define the \textit{$\ell$-th energetic moment} $\mathcal{E}_\ell(\G)$ of a graph $\G$ as
\begin{equation} \label{eq: energ moment}
	\mathcal{E}_\ell(\G) = \sum_{\lambda \in Spec(\G)} |\lambda|^\ell. 
\end{equation}
\end{defi}
Notice that $\mathcal{E}_\ell(\G) \ge 0$ and we have that $\mathcal{E}_0(\G)=|G|$ and $\mathcal{E}_1(\G) = \mathcal{E}(\G)$.

Let $\Gamma_1$ and $\Gamma_2$ be two graphs with the same number of vertices. The graphs $\Gamma_1$ and $\Gamma_2$ are said to be \textit{isospectral} if 
	$ Spec(\Gamma_1) = Spec(\Gamma_2)$ 
and are called \textit{equienergetic} if 
	$ \mathcal{E}(\Gamma_1) = \mathcal{E}(\Gamma_2) $.
It is clear 
that isospectrality implies equienergeticity, but the converse is in general false.

\subsection*{Outline and results}
Very briefly, in Section \ref{sec: family F} we recall from \cite{ChP} the family $\mathcal{F}$ of MDCGs and their structural and spectral properties. 
In Section \ref{sec: moments} we show that both the spectral and energetic moments of the MDCGs can be expressed in terms of the corresponding moments of the Cayley covers. 
Similarly, in Section \ref{sec: energy} we give the energy of the MDCGs in $\mathcal{F}$ 
in terms of the energy of the underlying Cayley graphs. 
In Section \ref{sec: hyp*energy} we study hypo/hyper-energeticity in $\mathcal{F}$, in terms of hypo/hyper-energeticity for the underlying Cayley graphs. 
In Sections \ref{sec: equienergy X, X+} and \ref{sec: equienergy} we study different kinds of equienergy between graphs in $\mathcal{F}$.

More precisely, we do the following. 
In Section \ref{sec: family F} we recall from \cite{ChP} the family $\mathcal{F}$ of mirror di-Cayley (sum) graphs 
	$$ \G_{G,S,T}^* = MX^*(G;S,T)$$ 
with $T=\{e\}, S$ or $S\cup \{e\}$.  In Proposition \ref{prop: directedness} we give their basic properties and in Theorem \ref{teo: prods} we show that the graphs $\G_{G,S,T}^*$ are products of the underlying Cayley graph $\G_{G,S}^*=X^*(G,S)$ and the 2-path $P_2$ or the 2-path with loops $\mathring{P}_2$, namely 
	$$ \G_{G;S,\{e\}}^*  = \G_{G,S}^* \Box P_2, \qquad  \G_{G;S,S}^*= \G_{G,S}^* \times \mathring{P_2}, \qquad \G_{G;S,S\cup\{e\}}^*  = \G_{G,S}^* \boxtimes P_2, $$
where $\Box$, $\times$ and $\boxtimes$ denote the Cartesian, the direct and the strong product, respectively. 
Then, in Proposition \ref{prop: spec bicayleys}, we recall the spectrum of each mirror di-Cayley (sum) graph $MX^*(G;S,T)$ in terms of the spectrum of the Cayley (sum) graph $X^*(G,S)$. Moreover, using the known description of the spectrum of Cayley graphs $X(G,S)$ and Cayley sum graphs $X^+(G,S)$ in terms of the irreducible characters $\chi$ of $G$, in Theorems \ref{thm: spec 3fam chars} and \ref{thm: spec 3fam chars sum} we give the spectrum of $MX(G;S,T)$ and $MX^+(G;S,T)$ in terms of the irreducible characters of $G$, respectively.

In Section \ref{sec: moments} we study the spectral and energetic $\ell$-th moments of MDCGs in the family $\mathcal{F}$. In Proposition \ref{prop: momentos} we give equalities for $M_\ell(\G^*_e)$, $M_\ell(\G^*_S)$ and $M_\ell(\G^*_{S \cup e})$ in terms of $M_\ell(\G^*)$, while in Proposition \ref{prop: momentos energeticos} we give equalities or bounds for $\mathcal{E}_\ell(\G^*_e)$, $\mathcal{E}_\ell(\G^*_S)$ and $\mathcal{E}_\ell(\G^*_{S \cup e})$ in terms of $\mathcal{E}_\ell(\G^*)$.

In Section \ref{sec: energy} we study the energy (i.e.\@ the 1-energetic moment) of the graphs 
$\G_e^*$, $\G_S^*$ and $\G_{S \cup e}^*$ in more detail. In Theorem \ref{teo: E(MDCG)=2E+cosas} we give the energy of these graphs in terms of the energy of the Cayley (sum) cover $\G^*$; namely,  
$$	\mathcal{E}(\G^*_S) = 2\mathcal{E}(\G^*) \qquad \text{and} \qquad 
\mathcal{E}(\G^*_T) \le 2 ( \mathcal{E}(\G^*) + |G|),$$ 
for $T=\{e\}$ or $T=S\cup \{e\}$.
However, if the spectrum of $\G$ is real we can give exact expressions for the energies in terms of $\mathcal{E}(\G^*)$ restricted to intervals. In Proposition \ref{prop: integral equienergy} we get neat expressions for the energies $\mathcal{E}(\G^*_e)$, $\mathcal{E}(\G^*_S)$, $\mathcal{E}(\G^*_{S \cup e})$ when the Cayley (sum) cover $\G^*$ is integral. 
Indeed, we show that the energy es even given by 
	$$\mathcal{E}(\G^*_T) = 2\mathcal{E}(\G^*) + 2 f(T),$$ 
where $f(T)$ is $0$, $m_0^*$ or $s_0^*$ when $T$ is respectively $S$, $\{e\}$ or $S\cup \{e\}$ (here $m_0^*$ is the multiplicity of the eigenvalue 0 and $s_0^*$ is the number of non-negative eigenvalues).  

In the next section we study hypoenergetic, orderenergetic and hyperenergetic MDCGs in the family $\mathcal{F}$.
In Propositions \ref{prop: hypo-hyper GS}, \ref{prop: hypo-hyper Ge} and \ref{prop: hypo-hyper GSe} we give conditions for the hypo/order/hyper-energeticity of $\G_e^*$, $\G_S^*$ and $\G_{S \cup e}^*$  in terms of the corresponding one of $\G^*$, respectively.

Then, in Section \ref{sec: equienergy X, X+} we address the topic of crossed equienergy in $\mathcal{F}$, 
that is equienergy between sum and no-sum di-Cayley graphs. 
More precisely, in Propositions \ref{prop: equienergy C1}, \ref{prop: equienerg no-iso C2} and \ref{prop: equienergy X,X+ case 3} we show that under some mild conditions,  
$\{\G_T,\G^+_T\}$ are equienergetic and non-isospectral graphs if and only if $\{\G,\G^+\}$ are equienergetic and non-isospectral graphs, for $T =\{e\}$, $S$ or $S \cup \{e\}$ respectively.

In the last section we deal with other types of equienergy. In this case, equienergetic pairs are more rare to find in general.
In first place, in Subsection \ref{subsec: mixed equienergy} we consider equienergy for MDCGs in the family $\mathcal{F}$ having the same covers but different di-connection sets, that is
equienergy between the graphs $\G_T^* = MX^*(G;S,T)$ and $\G_{T'}^* = MX^*(G;S,T')$ 
for $T,T'\in \{\{e\},S,S\cup\{e\}\}$ with $T\ne T'$ (see Proposition \ref{prop: equinergy TyT'} for a positive result and Proposition \ref{prop: no equinergy TyT'} for a negative result).
Finally, in Subsection \ref{subsec: general equienergy} we study the more general situation, which is equienergy between MCDGs having different covers. That is, we look for equienergy between graphs $MX^*(G_1,S_1,T_1)$ and $MX^*(G_2,S_2,T_2)$, where $T_1$ and $T_2$ are of the same type, i.e.\@ $T_i=S_i$, $T_i=\{e_i\}$ or $T_i=S_i \cup \{e_i\}$ for $i=1,2$. 
The different situations are covered in Propositions \ref{prop: equienergy dif covers C1}, \ref{prop: equienergy dif covers C2} and \ref{prop: equienergy dif covers C3}. As a general rule, again we have that under certain mild conditions, the pairs of MDCGs are equienergetic if and only if the underlying Cayley covers are equienergetic.

\section{The graphs in the family $\mathcal{F}$ and their spectra} \label{sec: family F}
Here we recall, from Sections 2 and 3 in \cite{ChP}, the definition of the mirror di-Cayley (sum) graphs and the particular family $\mathcal{F}$, as well as the basic structural properties (in particular Cayley structure and products decompositions) and the spectrum of the graphs in $\mathcal{F}$.

\begin{defi} \label{def: dicayleys}
Let $G$ be a group and $S,T \subset G$.
A \textit{mirror di-Cayley graph} 
    $ MX(G;S,T) $ 
is the graph having vertex set $G\times \Z_2$ and where there is a directed edge from $(h,i)$ to $(g,j)$ if and only if $j=i$ and $gh^{-1} \in S$ or $j=i+1$ and $gh^{-1} \in T$.   
That is,  
\begin{equation} \label{eq: dicays}
  \begin{aligned}
	(g,0) \sim (h,0) \quad \text{and} \quad (g,1) \sim (h,1) \qquad \Leftrightarrow \qquad gh^{-1} \in S, \\
   	(g,0) \sim (h,1) \quad \text{and} \quad (g,1) \sim (h,0) \qquad \Leftrightarrow \qquad gh^{-1} \in T.
  \end{aligned} 
\end{equation}
\textit{Mirror di-Cayley sum graphs} $MX^+(G;S,T)$ are defined similarly, changing $gh^{-1}$ by $gh$ above. 

\noindent 
\textsc{Notation:} When considering mirror di-Cayley graphs and mirror di-Cayley sum graphs together, we will denote them by $MX^*(G;S,T)$ and use the abbreviation MDCGs.
\end{defi}

\subsection*{The family $\mathcal{F}$ and its properties}
From now on, we will consider the following three kind of MDCGs (and very often the short notations):
\begin{equation} \label{eq: GeSSue}
	\begin{gathered} 
		\G^*_e := MX^*(G;S,\{e\}), \qquad \G^*_S := MX^*(G;S,S), \\ 
     	\G^*_{S\cup e} := MX^*(G;S,S \cup \{e\}).    
	\end{gathered}
\end{equation}
That is, the graphs $MX^*(G;S,T)$ with di-connection set $T \in \mathcal{S}$, where 
\begin{equation} \label{eq: S}
    \mathcal{S} = \big\{ \{e\}, S, S\cup \{e\} \big\}. 
\end{equation}
We will make use of the shorter notation in \eqref{eq: GeSSue} throughout the text. Also, for simplicity we will write $\G^*_e$ and $\G^*_{S\cup e}$ instead of $\G^*_{\{e\}}$ and $\G^*_{S\cup \{e\}}$, respectively.

Since $\G^*_{S \cup e} = \G^*_S$ if $e\in S$, from now on we will assume that $S$ is not a subgroup and that $e\notin S$ for the graph $\G^*_{S \cup e}$, unless explicit mention to the contrary.
We denote this family of MDCGs by $\mathcal{F}$. That is, in the notation of \eqref{eq: GeSSue} we have 
\begin{equation} \label{eq: famili F}
    \mathcal{F} = \{ \G^*_{e}, \G^*_{S}, \G^*_{S\cup e} : \text{$G$ a group, $S \subset G$}  \}. 
\end{equation}

\subsubsection*{Structural properties}
A subset $S$ of a group $G$ is \textit{symmetric} if it is closed under inversion, that is  
    $S=S^{-1}$, 
and $S$ is \textit{antisymmetric} if 
    $S \cap S^{-1} = \varnothing$. 
Also, $S$ is called \textit{normal} if for every $gh\in S$ we have that $hg \in S$. 
This happens if and only if 
    $gSg^{-1} =S \quad \text{for every $g\in G$}$, 
i.e.\@ $S$ is closed under conjugation. 
On the other hand, we say that $S$ is \textit{antinormal}, if 
$S \cap N_G(S) =\varnothing$,
where 
    $N_G(S) = \{g\in G : gSg^{-1} =S \}$ is the normalizer of $S$ in $G$.
Of course, if $G$ is abelian then any $S$ is closed under conjugation.

From Propositions 2.2 and 2.3 in \cite{ChP} we have the following 
structural properties of MDCGs.

\goodbreak     
\begin{prop} \label{prop: directedness} 
Let $G$ be a group and $S \subset G$. 
Consider the mirror di-Cayley graph $\G = MX(G;S,T)$ and the mirror di-Cayley sum graph $\G^+ = MX^+(G;S,T)$, where $T=\{e\}, S$ or $S\cup \{e\}$. Then, we have: 
		
\noindent 
$(a)$ \textsc{Directedness}.
The graph $\G$ is undirected (resp.\@ directed) if and only if $S$ is symmetric (resp.\@ antisymmetric). 
The graph $\G^+$ is undirected (resp.\@ directed) if and only if $S$ is normal (resp.\@ antinormal). 
		
\noindent 
$(b)$ \textsc{Loops}.
The graph $\G$ has loops at every vertex if and only if $e\in S$.
The graph $\G^+$ has loops at $(x,0)$ and $(x,1)$ if and only if $x^2\in S$. This happens for instance for the elements $x$ of order 2 in $S$ if $e\in S$.
In particular, if $S$ is a subgroup of $G$ then both $\G$ and $\G^+$ are looped.
		
\noindent 
$(c)$ \textsc{Regularity}.
The graphs $\G^*_e$ are $(s+1)$-regular, the graphs $\G^*_S$ are $(2s)$-regular and the graphs $\G^*_{S\cup e}$ are $(2s+1)$-regular, where $|S|=s$.

\noindent 
$(d)$ \textsc{Cayley structure}.
If $\Z_2=\{0,1\}$, then
\begin{equation} \label{eq: bicay=cay}
MX^*(G;S,T) = X^*\big( G\times \Z_2, (S\times \{0\}) \cup (T \times \{1\}) \big). 
\end{equation}
In particular, $MX^*(G;S,S) = X^* ( G \times \Z_2, S \times \Z_2 )$.
\end{prop}

\subsubsection*{Product decompositions}
Another nice structural property of these graphs is that they can be decomposed as basic products of graphs;
the mirror di-Cayley (sum) graphs $MX^*(G;S,T)$ are products between the underlying Cayley (sum) graph $X^*(G,S)$ and 2-path graphs. 
	
\begin{thm} \label{teo: prods}
Let $G$ be a group and $S$ a subset of $G$. We have that
	\begin{equation} \label{eq: prod}
		\begin{aligned}
			MX^{*}(G;S,\{e\}) 		&= X^{*}(G,S) \Box P_2,\\
			MX^{*}(G;S,S) 			&= X^{*}(G,S)\times \mathring{P_2},\\
			MX^{*}(G;S,S\cup\{e\})  &= X^{*}(G,S) \boxtimes P_2,
		\end{aligned}
	\end{equation}
where $P_2$ and $\mathring{P_2}$ are the $2$-path and the looped $2$-path, and $\Box$, $\times$ and $\boxtimes$ are the Cartesian, the direct (Kronecker) and the strong product of graphs, respectively.
\end{thm}

\begin{proof}
See Theorem 2.4 in \cite{ChP}.
\end{proof}

\subsection*{The spectrum of graphs in $\mathcal{F}$}
Here we recall expressions for the spectrum of MDCGs obtained in \cite{ChP}.

\subsubsection*{In terms of underlying Cayley graphs} 
Using the product structure of the graphs in $\mathcal{F}$, one can obtain the  spectrum of $MX^*(G;S,T)$ for $T=\{e\}, S$ and $S \cup \{e\}$ in terms of the spectrum of $X^*(G,S)$.
	
\begin{prop} \label{prop: spec bicayleys} 
Let $G$ be a group, $S \subset G$ and $\G^*=X^*(G,S)$. 
In the notations of \eqref{eq: GeSSue}, if $\text{Spec}(\G^*)=\{[\lambda_i^*]^{m_i^*}\}_{i\in I}$ then we have:
\begin{equation} \label{eq: specbi}
\begin{aligned}
Spec(\G^*_e) &= \{[\lambda_i^{*} + 1]^{m_i^*},[\lambda_i^*-1]^{m_i^*}\}_{i\in I},\\[1mm]
Spec(\G^*_{S}) &= \{[2\lambda_i^{*}]^{m_i^*}\}_{i\in I} \cup \{[0]^{|G|}\}, \\[1mm]
Spec(\G^*_{S\cup \{e\}}) &= \{[2\lambda_i^{*}+1]^{m_i^*}\}_{i\in I} \cup \{[-1]^{|G|}\}.
\end{aligned}
\end{equation}
\end{prop}

\begin{proof}
See Proposition 3.1 in \cite{ChP}.
\end{proof}

\subsubsection*{In terms of characters of the group}    
For $S$ a subset of $G$ and $\chi$ a character of $G$ put
\begin{equation} \label{eq: chiS}
    \chi(S) = \sum_{s\in S}\chi(s).    
\end{equation}

Using the previous proposition and the spectrum for Cayley graphs and Cayley sum graphs, in Section 3 of \cite{ChP} we have obtained the spectrum for the mirror di-Cayley (sum) graphs, that we now recall in two theorems.

\begin{thm} \label{thm: spec 3fam chars}
Let $G$ be a group and $S$ a normal subset of $G$. 
Then, the spectra of the di-Cayley graphs in $\mathcal{F}$ are given by: 
    \begin{equation} \label{eq: spec di chars}
	\begin{aligned}
		Spec(MX(G;S,\{e\})) &= \left\{ \left[  \tfrac{\chi(S)}{\chi(1)} \pm 1 \right]^{m_\chi} \right\}_{\chi \in \hat G},\\[1mm]
				Spec(MX(G;S,S)) &= \left\{ \left[ 2\tfrac{\chi(S)}{\chi(1)} \right]^{m_\chi} \right\}_{\chi \in \hat G} \cup \{[0]^{|G|}\}, \\[1mm]
				Spec(MX(G;S,S\cup\{e\})) &= 
				\left\{ \left[ 2\tfrac{\chi(S)}{\chi(1)}+1 \right]^{m_\chi} \right\}_{\chi \in \hat G} \cup \{[-1]^{|G|}\},
		\end{aligned}
    \end{equation}    
\nopagebreak 
where $m_\chi$ is the multiplicity of the eigenvalue $\lambda_\chi$ in each corresponding graph. 
\end{thm}
 
\begin{proof}
See Theorem 3.5 in \cite{ChP}.
\end{proof}

\begin{thm} \label{thm: spec 3fam chars sum}
Let $G$ be a group and $S$ a normal subset of $G$. Then, the spectrum of the di-Cayley sum graphs in $\mathcal{F}$ is real and given by: 
\begin{enumerate}[$(a)$]
	\item If $G$ is abelian, then we have  
		\begin{equation} \label{eq: spec di+ ab}
			\begin{aligned}
					Spec(MX^{+}(G;S,\{e\})) &= \left\{ \left[  \pm |\chi(S)| +1 \right]^{m_\chi}, \left[  \pm |\chi(S)| -1 \right]^{m_\chi} \right\}_{\chi \in \hat G}, \\ 
					Spec(MX^{+}(G;S,S)) &= 
					\{[\pm 2|\chi(S)| ]^{m_\chi}\}_{\chi \in \hat G} \cup \{[0]^{|G|}\}, \\
					Spec(MX^{+}(G;S,S\cup\{e\})) &= 
					\{[\pm 2|\chi(S)|+1]^{m_\chi}\}_{\chi \in \hat G} \cup \{[-1]^{|G|}\}.
			\end{aligned}
		\end{equation}    
			
	\item If $G$ is not abelian and $S$ is also symmetric, then we have
		\begin{equation} \label{eq: spec di+ non-ab}
			\begin{aligned}
					Spec(MX^{+}(G;S,\{e\})) &= \left\{ \left[  \pm \tfrac{\chi(S)}{\chi(1)} +1 \right]^{m_\chi}, \left[  \pm \tfrac{\chi(S)}{\chi(1)}-1 \right]^{m_\chi} \right\}_{\chi \in \hat G}, \\ 
					Spec(MX^{+}(G;S,S)) &= 
					\left\{ \left[ \pm 2\tfrac{\chi(S)}{\chi(1)} \right]^{m_\chi} \right\}_{\chi \in \hat G} \cup \{[0]^{|G|} \}, \\
					Spec(MX^{+}(G;S,S\cup\{e\})) &= 
					\left\{ \left[ \pm 2 \tfrac{\chi(S)}{\chi(1)} +1 \right]^{m_\chi} \right\}_{\chi \in \hat G} \cup \{[-1]^{|G|}\}.
			\end{aligned}
		\end{equation}   
	\end{enumerate}
In both $(a)$ and $(b)$, $m_\chi$ is the multiplicity of the eigenvalue $\lambda_\chi$ in each corresponding graph. 
\end{thm}

\begin{proof}
See Theorem 3.6 in \cite{ChP}.
\end{proof}

\section{Spectral and energetic moments} \label{sec: moments}
In this section we compute the spectral and energetic $\ell$-th moments of the di-Cayley graphs $MX^*(G;S,T)$ for  $T\in\{\{e\},S,S\cup\{e\}\}$. 

We recall from the introduction that for any $\ell \in \N_0$, the $\ell$-th spectral moment $M_\ell(\G)$ and the $\ell$-th energetic moment $\mathcal{E}_\ell(\G)$ of a graph $\G$ are defined as
\begin{equation} \label{eq: spec and energ moments}
	M_\ell(\G) = \sum_{\lambda \in Spec(\G)} \lambda^\ell \qquad \text{and} \qquad 	\mathcal{E}_\ell(\G) = \sum_{\lambda \in Spec(\G)} |\lambda|^\ell. 
\end{equation}

\subsection{Spectral moments} 
We now show that the spectral moments of the mirror di-Cayley (sum) graphs $MX^*(G;S,T)$ in the family $\mathcal{F}$ can be expressed in terms of the corresponding spectral moment of the cover Cayley graph $X^*(G,S)$.

\begin{prop} \label{prop: momentos}
Let $G$ be a group, $S \subset G$ and $\G^*=X^*(G,S)$. 
In the notations of \eqref{eq: GeSSue}--\eqref{eq: S}, for any $\ell \in \N$ we have the following:
\begin{enumerate}[$(a)$]
	\item The spectral moments of $\G^*_S$ are given by
$$ M_\ell(\G^*_S) = 2^\ell M_\ell( \G^*).$$

	\item The spectral moments of $\G^*_e$ are given by
		$$ 
			M_\ell(\G^*_e) = \begin{cases}
			2 \sum\limits_{0 \le j \, {\rm even} \, \le \ell} \binom{\ell}{j} M_j(\G^*), &  \qquad \text{if $\ell$ is even}, \\[3mm] 
			2 \sum\limits_{0 \le j \, {\rm odd} \, \le \ell} \binom{\ell}{j} M_j(\G^*), &  \qquad \text{if $\ell$ is odd}.\end{cases} 
		$$

	\item The spectral moments of $\G^*_{S \cup e}$ are given by
		$$ M_\ell(\G^*_{S \cup e}) = \Big( \sum_{j=0}^\ell 2^j \tbinom{\ell}{j} M_j(\G^*) \Big) + (-1)^\ell |G|. $$
\end{enumerate}
\end{prop}

\begin{proof}
Assume that $Spec(\G^*)=\{[\lambda_i^*]^{m_i^*}\}_{i\in I}$ for some finite index set $I$.	

\sk 
\noindent ($a$) 
For $T=S$, by \eqref{eq: spec and energ moments} and Proposition \ref{prop: spec bicayleys}, we have that 
$$M_\ell(\G^*_S) = \sum_{\lambda^* \in Spec(\G^*_S)} (\lambda^*)^\ell = \sum_{i \in I} (2\lambda_i^*)^\ell = 2^\ell M_\ell(\G^*).$$

\sk 
\noindent ($b$) 
For $T=\{e\}$, by \eqref{eq: spec and energ moments} and Proposition \ref{prop: spec bicayleys}, we have that
	$$	M_\ell(\G^*_e) = \sum_{\lambda^* \in Spec(\G^*_e)} (\lambda^*)^\ell  
					   = \sum_{i \in I} (\lambda_i^*+1)^\ell + \sum_{i \in I} (\lambda_i^*-1)^\ell. $$
By the binomial theorem we have 
	$$ 
	(\lambda_i^*+1)^\ell = \sum_{j=0}^\ell \tbinom{\ell}{j}(\lambda_i^*)^j \qquad \text{ and } \qquad (\lambda_i^*-1)^\ell = \sum_{j=0}^\ell (-1)^{\ell-j} \tbinom{\ell}{j}(\lambda_i^*)^j,
	$$ 
and therefore we get
	$$
		M_k(\G^*_e) = \sum_{j=0}^\ell (1+(-1)^{\ell-j}) \tbinom{\ell}{j} \sum_{i \in I} (\lambda_i^*)^j 
		= \sum_{j=0}^\ell (1+(-1)^{\ell-j}) \tbinom{\ell}{j} M_j(\G^*).
	$$ 
		
Now, if $\ell$ is even (resp.\@ odd), then 
	$$ (1+(-1)^{\ell-j}) = 
		\begin{cases}
			2 & \quad \text{if $j$ is even (resp.\@ odd)}, \\[1mm]
			0 & \quad \text{if $j$ is odd (resp.\@ even)}, \end{cases} 
	$$
and hence, 
	$$ M_\ell(\G^*_e) = 2 \sum_{\substack{\begin{smallmatrix}
		0 \le j \text{even}  \le \ell \\
			\text{(resp.\@ $j$ odd)}
		\end{smallmatrix}}} \tbinom{\ell}{j} M_j(\G^*) ,$$
as we wanted to show.

\sk 
\noindent ($c$) 
Finally, in the case of $T=S\cup\{e\}$, by \eqref{eq: spec and energ moments} and Proposition \ref{prop: spec bicayleys}, we have that 
	$$
		\begin{aligned}
			M_\ell(\G^*_{S\cup e}) &= \sum_{\lambda^* \in Spec(\G^*_{S \cup e})} (\lambda^*)^\ell = \Big( \sum_{i \in I} (2\lambda_i^*+1)^\ell \Big) + (-1)^\ell|G| \\
			&= \sum_{i \in I} \Big( \sum_{j=1}^\ell \tbinom{\ell}{j}(2\lambda_i^*)^j \Big) + (-1)^\ell|G| 
			= \sum_{i=1}^\ell 2^j \tbinom{\ell}{j} M_j(\G^*) + (-1)^\ell|G|.
		\end{aligned}
	$$
The proof is thus complete.
\end{proof}
 
\begin{rem}
Any MDCG $MX^*(G;S,T)$ has vertex set $G\times \Z_2$, hence of cardinal $2|G|$.  
One can check from the expressions in Proposition \ref{prop: momentos} that 
$ M_0(\G^*_T) = 2|G| $ and $M_1(\G^*_T) = 0$ 
for $T=\{e\}, S, S\cup \{e\}$, as it should be. 
\end{rem}

\subsection{Energetic moments} 
We now exhibit two different kinds of upper bounds for the energetic moments of the mirror di-Cayley (sum) graphs $\G^*_T = MX^*(G;S,T)$ in the family $\mathcal{F}$.

First, for each $T \in \{S,\{e\}, S\cup \{e \} \}$, we give bounds for the $\ell$-energetic moment of $\G^*_T$ in terms of the corresponding $\ell$-energetic moment of the cover Cayley graph $\G^*=X^*(G,S)$.

\begin{prop} \label{prop: momentos energeticos}
Let $G$ be a group, $S \subset G$, and let $\G^*=X^*(G,S)$. 
Using the notation from \eqref{eq: GeSSue}--\eqref{eq: S}, for any $\ell \in \N$, we have:
\begin{enumerate}[$(a)$]
			
	\item The $\ell$-th spectral moment of $\G^*_S$ is given by
		$ \mathcal{E}_\ell(\G^*_S) = 2^\ell \mathcal{E}_\ell( \G^*) $. \sk 
			
	\item The $\ell$-th spectral moment of $\G^*_e$ satisfies
		$ \mathcal{E}_\ell(\G^*_e) \le 2^\ell (\mathcal{E}_\ell(\G^*)+|G|) $. \sk 
			
	\item The $\ell$-th spectral moment of $\G^*_{S \cup e}$ satisfies
		$ \mathcal{E}_\ell(\G^*_{S \cup e}) \le 2^{2\ell-1} \mathcal{E}_\ell(\G^*)+ (2^{\ell-1}+1)|G| $.
\end{enumerate}
\end{prop}

\begin{proof}
	Item $(a)$ follows directly from the eigenvalues of $\G^*_S$ obtained in Proposition~\ref{prop: spec bicayleys}. To prove items $(b)$ and $(c)$, we use the eigenvalue formulas for $\G^*_e$ and $\G^*_{S \cup e}$ from \eqref{eq: specbi} together with the triangle inequality. Specifically, applying the standard bound $$ |x + y|^\ell \le 2^{\ell-1} (|x|^\ell + |y|^\ell) $$ for $x,y \in \C$ to the eigenvalues, and summing over the spectrum, leads to the desired upper bounds. 
\end{proof}

Now, for each $T \in \{S,\{e\}, S\cup \{e \} \}$ we give bounds for the $\ell$-energetic moment $\mathcal{E}_\ell(\G^*_T)$ in terms of all the $j$-energetic moments $\mathcal{E}_j(\G^*)$ of the cover Cayley graph $X^*(G,S)$, for every $j=0,\ldots,\ell$.

\begin{prop} \label{prop: momentos energeticos otras cotas}
Let $G$ be a group, $S \subset G$ and $\G^*=X^*(G,S)$. 
	In the notations of \eqref{eq: GeSSue}--\eqref{eq: S}, for any $\ell \in \N$ we have the following:
	\begin{enumerate}[$(a)$]
		\item The $\ell$-energetic moment of $\G^*_e$ 
		yields 
		$$ 
		\mathcal{E}_\ell(\G^*_e) \le \begin{cases}
			2 \sum\limits_{0 \le j \, {\rm even} \, \le \ell} \binom{\ell}{j} \mathcal{E}_j(\G^*), &  \qquad \text{if $\ell$ is even}, \\[3mm] 
			2 \sum\limits_{0 \le j \, {\rm odd} \, \le \ell} \binom{\ell}{j} \mathcal{E}_j(\G^*), &  \qquad \text{if $\ell$ is odd}.\end{cases} 
		$$
		
		\item The $\ell$-energetic moment of $\G^*_{S \cup e}$ satisfies
		$$ \mathcal{E}_\ell(\G^*_{S \cup e}) \le \Big( \sum_{j=0}^\ell 2^j \tbinom{\ell}{j} \mathcal{E}_j(\G^*) \Big) + |G|. $$
	\end{enumerate}
\end{prop}

\begin{proof}
	This proof is similar to the one of the previous proposition. In fact, one has to apply triangle inequality because of the presence of absolute values and combine this with an application of the binomial theorem.  
\end{proof}

\section{Energy} \label{sec: energy}
In the previous section we have expressed the spectral and energetic moments $M_\ell(\G^*_T)$ and $\mathcal{E}_\ell(\G^*_T)$ of the MDCGs in the family $\mathcal{F}$ in terms of the corresponding moments of the Cayley covers $X^*(G,S)$. 
Here, and from now on, we concentrate on the case $\ell=1$. Also, since $M_1(\G)=0$ and $\mathcal{E}_1(\G)$ is the energy, we will consider only energetic problems.

Thus, in this section we compute the energy of the mirror di-Cayley (sum) graphs 
$MX^*(G;S,T)$ with $T \in \mathcal{S}=\{\{e\}, S, S\cup \{e\}\}$ in terms of the energy of the underlying Cayley (sum) graph $X^*(G,S)$, in more detail.

\subsubsection*{Some notation}
We recall that the energy of a graph $\G$ is defined by  
\begin{equation} \label{eq: energy2}
	\mathcal{E}(\G) = \sum_{\lambda \in Spec(\G)} |\lambda| = \sum_{i=1}^s m_i |\lambda_i|,
\end{equation}
where $Spec(\G) = \{ [\lambda_1]^{m_1}, \ldots, [\lambda_s]^{m_s} \}$.

For ease, if $(G,S)$ is fixed, it will be useful to denote the spectrum of the graph $\G^*=X^*(G,S)$ simply by $\frak{S}^*$, that is 
\begin{equation} \label{eq: S*}
	\frak{S}^* = Spec(X^*(G,S)) \subset \C.
\end{equation} 
If $\mathcal{I} \subset \R$ is an interval, 
we will use the shorthand notation 
\begin{equation} \label{eq: S_I*}
	\frak{S}^*_\mathcal{I} =\frak{S}^* \cap \mathcal{I} \subset \R
\end{equation}
for the spectrum restricted to $\mathcal{I}$ and, although $\frak{S}^*$ is a multiset, we will write 
	$ \# \frak{S}^*_\mathcal{I} $  
to denote the number of eigenvalues of $X^*(G,S)$ in $\mathcal{I}$, counted with multiplicities.
Furthermore, we define the \textit{restricted energy} of $\G^* = X^*(G,S)$ to the interval $\mathcal{I}$ by 
\begin{equation} \label{eq: restricted energy}
	\mathcal{E}_\mathcal{I}(\G^*) = \sum_{\lambda^* \in \frak{S}^*_\mathcal{I}} |\lambda^*|.
\end{equation}
This will allow us to give some expressions shorter and more conveniently.
In particular, we will need the following particular intervals
\begin{equation} \label{eq: intervals}
	\mathcal{J}=(-1,1), \qquad \mathcal{K}=[-\tfrac 12,0) \qquad \text{and} \qquad \mathcal{K}_\infty = \mathcal{K} \cup \R_{\ge 0} =[-\tfrac 12, \infty).
\end{equation}

\subsection{The general case}
Here we give the energies of the mirror di-Cayley (sum) graphs 
$MX^*(G;S,T)$ in terms of the energy and restricted energies of the underlying Cayley graph $X^*(G,S)$, when possible.
In the case when $T\ne S$ we will have to assume that the spectrum of the Cayley graph is real. 

\begin{thm} \label{teo: E(MDCG)=2E+cosas}
Let $G$ be a group, $S\subset G$, $\G^*=X^*(G,S)$ and put $\frak{S}^*= \text{Spec}(\G^*)$.
Then, the energies of the graphs $\G_T^*$ with $T\in \mathcal{S} = \{\{e\}, S, S\cup \{e\}\}$ satisfy  
\begin{equation} \label{eq: Ebi=2E}
	\mathcal{E}(\G^*_S) = 2\mathcal{E}(\G^*) \qquad \text{and} \qquad 
	\mathcal{E}(\G^*_T) \le 2 ( \mathcal{E}(\G^*) + |G|), 
\end{equation}
for $T=\{e\}$ or $T=S\cup \{e\}$. 
Furthermore, if $\text{Spec}(\G) \subset \mathbb{R}$, 
then we have that 
\begin{equation} \label{eq: Ebi=2E+cosas}
\begin{aligned}
	\mathcal{E}(\G^*_e) & = 2 \{ \mathcal{E}(\G^*) - \mathcal{E}_{\mathcal{J}}(\G^*) + \# \frak{S}_{\mathcal{J}}^* \},  \\[.5mm]
	\mathcal{E}(\G^*_{S\cup e}) & = 2 \{ \mathcal{E}(\G^*) - 2\mathcal{E}_\mathcal{K}(\G^*) + \# \frak{S}_{\mathcal{K}_\infty}^* \},
	\end{aligned}
\end{equation}
in the notations of \eqref{eq: S*}--\eqref{eq: intervals}.
\end{thm}

\begin{proof}
The first part of the statement, namely the two expressions in \eqref{eq: Ebi=2E}, follows immediately from Proposition \ref{prop: momentos energeticos} by taking $\ell = 1$. 
	
Now, assume that $Spec(\G) \subset \mathbb{R}$ (and hence $Spec(\G^+) \subset \mathbb{R}$ also).
Let us prove the first equality in \eqref{eq: Ebi=2E+cosas}. By \eqref{eq: specbi}, since the eigenvalues of $\G^*_e$ are of the form $\lambda^* + 1$ and $\lambda^* - 1$, we have that
	$$\begin{aligned}
		\mathcal{E}(\G^*_e) &= \sum_{\lambda^*\in \frak{S}^*}|\lambda^*+1| + \sum_{\lambda^*\in \frak{S}^*}|\lambda^*-1| = \sum_{\lambda^*\in \frak{S}^*}(|\lambda^*+1|+|\lambda^*-1|).
	\end{aligned} $$
Note that 
	$$ |\lambda^*+1|+|\lambda^*-1| = \begin{cases} 
		2|\lambda^*|, & \qquad \text{if $|\lambda^*| \ge 1$}, \\[1mm]
		\hfil 2,			 & \qquad \text{if $|\lambda^*| < 1$}.
	\end{cases}$$
In this way, putting $\mathcal{J}=(-1,1)$, we have
$$ 	\mathcal{E}(\G^*_e) = \sum_{\lambda^*\in \frak{S}^* \smallsetminus \mathcal{J}} 2|\lambda^*| + \sum_{\lambda^*\in \frak{S}^* \cap \mathcal{J}} 2 = 2 \mathcal{E}(\G^*) - 2 \sum_{\lambda^*\in \frak{S}_\mathcal{J}^*}|\lambda^*| + 2 \#(\frak{S}^*_\mathcal{J}),
$$
as we wanted to show.

Now, we prove the second equality in \eqref{eq: Ebi=2E+cosas}. 
Since the eigenvalues of $\G^*_{S \cup e}$ are of the form $2\lambda^*+1$ and $-1$, by \eqref{eq: specbi}, we have that
\begin{equation} \label{eq: energyGSSSe}
	\mathcal{E}(\G^*_{S \cup e}) = \sum_{\lambda^*\in \frak{S}^*}|2\lambda^*+1| + \sum_{\lambda^*\in \frak{S}^*}|-1| = \sum_{\lambda^*\in \frak{S}^*}|2\lambda^*+1| + \#\frak{S}^*.    
\end{equation}

Now, observe that if $\lambda^*\geq0$ then $|2\lambda^*+1|=2|\lambda^*|+1$, so
	$$ |2\lambda^*+1| = \begin{cases}
		\hfil 2|\lambda^*|-1, & \qquad \text{if $\lambda^* < -\tfrac 12$}, \\[1mm]
		-2|\lambda^*|+1, & \qquad \text{if $-\tfrac 12 \le \lambda^* < 0$}, \\[1mm]
		\hfil 2|\lambda^*|+1, & \qquad \text{if $\lambda^* \ge 0$}.	\end{cases} $$
Thus, we have that 
	$$ \sum_{\lambda^*\in \frak{S}^*} |2\lambda^*+1| = \sum_{\lambda^*\in \frak{S}_{\mathcal{K}^-}^*} (2|\lambda^*|-1) + \sum_{\lambda^*\in \frak{S}^*_\mathcal{K}} (-2|\lambda^*|+1) + \sum_{\lambda^*\in \frak{S}^*_{\mathcal{K}^+}} (2|\lambda^*|+1), $$
where 
	$$ \mathcal{K}^-=(-\infty,-\tfrac{1}{2}), \qquad \mathcal{K} = [-\tfrac{1}{2},0) \qquad \text{and} \qquad \mathcal{K}^+ =[0,\infty) = \R_{\ge 0}.$$ 
That is, 
	$$ \sum_{\lambda^*\in \frak{S}^*} |2\lambda^*+1| = 2 \big\{ \mathcal{E}_{\mathcal{K}^-}(\G^*) - \mathcal{E}_\mathcal{K}(\G^*) + \mathcal{E}_{\mathcal{K}^+}(\G^*) \big\} - \# \frak{S}^*_{\mathcal{K}^-} + \# \frak{S}^*_\mathcal{K} + \# \frak{S}^*_{\mathcal{K}^+} $$
in the notations of \eqref{eq: S_I*} and \eqref{eq: restricted energy}.
In this way, by adding and subtracting 
$2 \mathcal{E}_\mathcal{K}(\G^*)$ above and using that $\mathcal{K}^- \cup \mathcal{K} \cup \mathcal{K}^+ = \R$, from \eqref{eq: energyGSSSe} we get that
$$ \begin{aligned}
	\mathcal{E}(\G^*_{S \cup e}) &= 2 \mathcal{E}(\G^*) - 4\mathcal{E}_\mathcal{K}(\G^*) + \# \frak{S}^* - \# \frak{S}^*_{\mathcal{K}^-} + \# \frak{S}^*_\mathcal{K} + \# \frak{S}^*_{\mathcal{K}^+},
\end{aligned} $$
from which the second expression of \eqref{eq: Ebi=2E+cosas} readily follows. 
\end{proof}

Thus, we see that the energy of $\G_S$ is two times the energy of $\G^*$ and, if $\G$ has real spectrum (i.e.\@ $\G$ is undirected), the energies of $\G_e$ and $\G_{S \cup e}$ are two times the energy of $\G^*$ plus something else.

\begin{coro} \label{coro: desigualdades energia}
Let $G$ be a group and $S$ a subset of $G$. If $Spec(X^*(G,S))\subset\R$, we have that
\begin{equation} \label{eq: energy bounds}
	\mathcal{E}(\G_S^*) \le \mathcal{E}(\G_T^*) \le 2\mathcal{E}(\G_S^*) +2|G|,
\end{equation} 
for $T=\{e\}$ or $T=S\cup \{e\}$.
\end{coro}

\begin{proof}
The upper bounds follow directly from \eqref{eq: Ebi=2E}. 

For the lower bound of $\G_e^*$ notice that, if $\lambda^* \in \frak{S}_{\mathcal{J}}^*$, then $|\lambda^*|<1$, and thus 
	$$\mathcal{E}_{\mathcal{J}}(\G^*) < \sum_{\frak{S}_{\mathcal{J}}^*} 1 = \# \frak{S}_{\mathcal{J}}^*.$$
Hence 
$\# \frak{S}_{\mathcal{J}}^*-\mathcal{E}_{\mathcal{J}}(\G^*) \ge 0$ and 
	$$\mathcal{E}(\G_S^*) = 2 \mathcal{E}(\G^*) \le 2 \mathcal{E}(\G^*) + 2 (\# \frak{S}_{\mathcal{J}}^*-\mathcal{E}_{\mathcal{J}}(\G^*)) = \mathcal{E}(\G_e^*).$$

In the case of $\G_{S\cup e}^*$, we have that if $\lambda^* \in \frak{S}_{\mathcal{K}}^*$, then $|\lambda^*|<\tfrac{1}{2}$. Thus, since the eigenvalue $|S|\in \frak{S}_{\mathcal{K}_\infty}^* \smallsetminus \frak{S}_{\mathcal{K}}^*$, we have that
	$ 2\mathcal{E}_\mathcal{K}(\G^*) < \#\frak{S}_{\mathcal{K}}^* < \#\frak{S}_{\mathcal{K}_\infty}^*$.
Then, 			
$ \# \frak{S}_{\mathcal{K}_\infty}^* - 2\mathcal{E}_\mathcal{K}(\G^*) \ge 0$ 
and
	$$\mathcal{E}(\G_S^*) = 2 \mathcal{E}(\G^*) \le  2 \mathcal{E}(\G^*) + 2(\#\frak{S}_{\mathcal{K}_\infty}^*-2\mathcal{E}_\mathcal{K}(\G^*)) = \mathcal{E}(\G_{S\cup e}^*).$$
Thus, both lower bounds hold, concluding the proof.
\end{proof}

\begin{rem}[NEPS of graphs and energy] \

\noindent 
$(i)$ 
By Proposition \ref{teo: prods}, the graphs $\G^*_T$ with $T\in \mathcal{S}=\{ \{e\}, S, S\cup \{e\}\}$ are cartesian, direct and strong products of $\G^*$ with $P_2$ or $\mathring{P}_2$. These products are NEPS of two graphs (see Section 2 of \cite{ChP} for more detalis).

\noindent 
($ii$) 
According to Theorem $3$ in \cite{Stevanovic2}, there exists a function that expresses the energy of a NEPS of $n$ graphs in terms of the energies of its factor graphs if and only if the NEPS corresponds to the direct product. This is reflected in Theorem \ref{teo: E(MDCG)=2E+cosas} for the case $n=2$. 
In fact, 
	$$ \mathcal{E}(\G_S^*) = \mathcal{E}(\G^* \times \mathring{P_2}) = \mathcal{E}(\G^*) \mathcal{E}(\mathring{P_2}) $$
and $Spec (\mathring{P_2}) = \{[0]^1, [2]^1\}$.
\end{rem}

\subsection{Integral energy}
When the underlying Cayley graph is integral, much neater expressions for the energies can be obtained. 
Namely, $\mathcal{E}(\G^*_T) = 2\mathcal{E}(\G^*) + 2 f(T)$, where $f(T)$ depends whether $T$ is $S$, $\{e\}$ or $S\cup \{e\}$. 

\begin{prop} \label{prop: integral equienergy}
Let $G$ be a group and $S \subset G$.
If $\G = X(G,S)$ is integral then the energies of the di-Cayley graphs $\G^*_S$, $\G^*_e$ and $\G^*_{S\cup e}$ are even and given by
\begin{equation} \label{eq: energyint}
	\begin{aligned}
		\mathcal{E}(\G^*_S) & = 2 \mathcal{E}(\G^*), \\[1mm]
		\mathcal{E}(\G^*_e) & = 2 \mathcal{E}(\G^*) + 2m_0^{*}, \\[1mm]
		\mathcal{E}(\G^*_{S\cup e}) &= 2\mathcal{E}(\G^*) + 2 s_{0}^{*},
	\end{aligned}
\end{equation}
where $m_0^{*}$ denotes the multiplicity of the zero eigenvalue in $\G^*$ 
and $s_{0}^{*} = \# (\frak{S}^* \cap \R_{\ge 0})$.
In particular, we have that 
$$ \mathcal{E}(\G^*_S) \le \mathcal{E}(\G^*_e) <  \mathcal{E}(\G^*_{S\cup e}).$$ 
Moreover, 
	$\mathcal{E}(\G_S) \equiv 0 \pmod 4$ and if $|G|$ is even, then 
	$\mathcal{E}(\G_e) \equiv 0 \pmod 4$.
In the case of the  mirror di-Cayley sum graphs $\G_S^+$ and $\G_e^+$ these congruences are valid provided that the Cayley graph $\G^+$ has no loops.
\end{prop}

\begin{proof}
First, note that if $\G=X(G,S)$ is integral then $\G^+=X^+(G,S)$ is integral.

\sk 

$\bullet$ The second equality in the statement always holds by \eqref{eq: Ebi=2E+cosas}.
To prove the first equality in \eqref{eq: energyint}, note that the only possible eigenvalue in $\mathcal{J}=(-1,1)$ is $0$, and thus we have that $\frak{S}^*_\mathcal{J} \subset\{0\}$. Hence, we obtain that 
	$$\mathcal{E}_{\mathcal{J}}(\G^*) = \sum_{\lambda^* \in \frak{S}^*_\mathcal{J}} |\lambda^*| = 0$$ 
and $\#\frak{S}^*_\mathcal{J}=m_0^*$. 
Then, $\mathcal{E}(\G^*_e) = 2\mathcal{E}(\G^*) + 2m_0^{*}$.

Finally, for the third equality in \eqref{eq: energyint}, notice that $\frak{S}^*_{\mathcal{K}} = \frak{S}^*\cap [-\tfrac{1}{2},0) = \varnothing$, so that $\mathcal{E}_{\mathcal{K}}(\G^*)=0$ 
and also 
	$$\# \frak{S}^*_{\mathcal{K}_\infty} = \# \frak{S}^*_{\R_{\ge 0}} = s_0^*.$$
In this way, we obtain that 
$\mathcal{E}(\G^*_{S\cup e}) = 2\mathcal{E}(\G^*) +2 s_{0}^{*}$, as we wanted to see.

\sk 

$\bullet$ The assertion regarding the bounds is clear from \eqref{eq: energyint}, since 
	$0 \le m_0^* < s_0^*$. 
The inequality holds because, being a regular graph, the degree of regularity of $\G^*_{S\cup e}$ is an eigenvalue.

\sk 

$\bullet$ To prove the last statement, recall that the sum of the eigenvalues of a graph is equal to the trace of its adjacency matrix. When the graph has no loops, the trace of the adjacency matrix equals zero; that is, if $Spec(\G^*) = \{\lambda_1^*,\dots, \lambda_n^*\}$, then 
	$$ \lambda_1 + \cdots + \lambda_n = \Tr(A_{\G^*}) = 0. $$
We denote by $\mathcal{S}^+$ the subset of all positive eigenvalues of $\G^*$ and $\mathcal{S}^-$ the subsets of all negative eigenvalues of $\G^*$. Then, $\sum_{\lambda^* \in \mathcal{S}^+} \lambda^* = -\sum_{\lambda^* \in \mathcal{S}^-} \lambda^*$.
Therefore,
	$$ \mathcal{E}(\G^*) = \sum_{\lambda^*\in \mathcal{S}^+} \lambda^* - \sum_{\lambda^* \in \mathcal{S}^-} \lambda^* = 2 \sum_{\lambda^* \in \mathcal{S}^+} \lambda^*,$$
and the energy of the Cayley (sum) graph $\G^*$ is even.
From this, it follows that 
	$$ \mathcal{E}(\G_S^*) = 2\mathcal{E}(\G^*) \equiv 0 \pmod 4.$$

Now, if $|G|$ is even, since the number of complex characters of $G$ is always even, we have that the number of real characters of $G$ is also even. On the other hand, the multiplicity of the eigenvalue $0$ in $\G^*$, $m_0^*$ is exactly the number of real characters of $G$, that is, $m_0^*$ is even.
Thus,
$$\mathcal{E}(\G_e^*) = 2 \mathcal{E}(\G^*) + 2m_0^{*} \equiv 0 \pmod 4.$$ 
Consequently, in both cases, the energy of the graph is a multiple of $4$.
Finally, if the sums graphs $\G_S^+$ and $\G_e^+$ have no loops, the same proof applies and the result follows.
\end{proof}

We now illustrate the proposition.

\begin{exam} \label{exam: CayZ4}
Let $G=\Z_4=\{0,1,2,3\}$ and $S=\Z_4^* = \{1,3\} \subset \Z_4$. The unitary Cayley graph $\G=X(G,S)$ equals the 4-cycle $C_4$, and has spectrum given by
    $Spec(\G) = \{ [2]^1, [0]^2, [-2]^1 \}$.
Then, by Proposition \ref{prop: spec bicayleys}, we have that
\begin{align*}
Spec(\G_S) & = \{ [4]^1, [0]^6, [-4]^1\}, \\ 
Spec(\G_e) & = \{ [3]^1,[1]^3,[-1]^3,[-3]^1 \}, \\ 
Spec(\G_{S \cup e}) & = \{[5]^1, [1]^2, [-1]^4,[-3]^1\}.    
\end{align*}   
It is immediate to check that $\mathcal{E}(\G)=4$ and 
\begin{align*}
\mathcal{E}(\G_S) & = 2\cdot 4 + 6 \cdot 0 = 8, \\ 
\mathcal{E}(\G_e) & = 2\cdot 3 + 6 \cdot 1 = 12, \\  
\mathcal{E}(\G_{S \cup e}) & = 5+3+ 6\cdot 1 = 14,   
\end{align*}
from where we have that $\mathcal{E}(\G_S) < \mathcal{E}(\G_e) < \mathcal{E}(\G_{S\cup e})$ and 
	$$ \mathcal{E}(\G_S) = 2 \mathcal{E}(\G), \qquad \mathcal{E}(\G_e) = 2 \mathcal{E}(\G) + 2 m_0, \qquad 
	\text{and} \qquad \mathcal{E}(\G_{S \cup e})= 2\mathcal{E}(\G) + 2s_0, $$
where $m_0=2$ and $s_0=3$, in accordance with Proposition \ref{prop: integral equienergy}.
\hfill $\diamond$
\end{exam}

\begin{rem}
Proposition \ref{prop: integral equienergy} establishes that if $\G^*$ is an integral Cayley graph, then the three families considered also have integer energy. It is worth noting that the integrality of the spectrum of $\G^*$ is a crucial assumption. Indeed, $\G^*$ might have integer energy without being an integral graph; in such cases, while the energy of $\G_S^*$ remains an integer, the energies of the corresponding families $\G_e^*$ and $\G_{S\cup e}^*$ are not guaranteed to be integers.
\end{rem}

To support the previous remark, consider the following example.

\begin{exam}
The spectrum of the Cayley graph $\G=X(\Z_8,S)$ with $S=\{1,2,3,5,7\}$ is given by
	$ Spec(\G) = \{[5]^1,[-1]^2,[-3]^1,[i]^2,[-i]^2 \} \subset \Z[i] \smallsetminus \Z$,
but has integral energy $\mathcal{E}(\G) = 14$.
From this, the spectrum of the mirror di-Cayley graphs $\G_S$, $\G_e$ and $\G_{S\cup e}$ are given by
	\begin{align*}
		Spec(\G_S) &= \{[10]^1,[0]^8,[-2]^2,[-6]^1,[2i]^2,[-2i]^2\},\\
		Spec(\G_e) &= \{[6]^1,[4]^1,[0]^2,[-2]^3,[-4]^1,[i\pm 1]^2,[-i\pm 1]^2\},\\
		Spec(\G_{S\cup e}) &= \{[11]^1,[-1]^{10},[-5]^1,[2i+1]^2,[-2i+1]^2\},
	\end{align*}	
which are not integral. It is immediate to check that the energies of these graphs are
	$$ \mathcal{E}(\G_S) = 28 \in \Z, \qquad \mathcal{E}(\G_e) = 18+8\sqrt{2} \notin \Z, \qquad 
	\mathcal{E}(\G_{S\cup e}) = 26 + 4\sqrt{5} \notin \Z.$$
As anticipated, the integer energy is preserved for $\G_S^*$, but fails for $\G_e^*$ and $\G_{S\cup e}^*$ due to their non-integral spectra.
\hfill $\diamond$ 
\end{exam}

\begin{rem}
There are numerous examples of integral Cayley graphs whose spectra have been completely determined. These include broad general families of graphs --such as hypercubes, complete multipartite graphs, and certain unitary Cayley graphs (see for instance the survey paper \cite{LZ} or Examples~4.11--4.14 in \cite{ChP})-- as well as specific non-trivial instances over small-order groups, like those cataloged in the tables in \cite{ChP3}. For any of these integral graphs, Proposition \ref{prop: integral equienergy} guarantees that the energies of their associated mirror di-Cayley graphs $\Gamma_S^*$, $\Gamma_e^*$, and $\Gamma_{S \cup e}^*$ can be explicitly calculated and will always yield integer values.
\end{rem}

\section{Hypo-, order- and hyperenergetic MDCGs in $\mathcal{F}$} \label{sec: hyp*energy}

In this section we provide necessary and sufficient conditions for the graphs $\G_e^*$, $\G_S^*$ and $\G_{S\cup e}^*$ to be  (weak) hyperenergetic, orderenergetic or hypoenergetic. 

We begin by recalling the relevant definitions. A graph $\G$ with $n$ vertices is called
\begin{itemize}
	\item \textit{hypoenergetic} if $\mathcal{E}(\G)<n$, \smallskip
	
	\item \textit{orderenergetic} if $\mathcal{E}(\G)=n$, \smallskip

	\item \textit{hyperenergetic} if $\mathcal{E}(\G)>2n-2$. 
\end{itemize}

Additionally, we will say that $\G$ is 
\begin{itemize}
	\item \textit{weak hyperenergetic} if $\mathcal{E}(\G)>2n-1$.
\end{itemize}

Notice that for trivial Cayley (sum) graphs $\G^*=X^*(G,\varnothing)$ we have that $Spec(\G^*)=\{[0]^{|G|}\}$,  and hence $\G^*$ is hypoenergetic since $\mathcal{E}(\G^*) = 0$.

We first compare energetic properties between different mirror di-Cayley (sum) graphs.

\begin{prop}
	Let $G$ be a group, $S \subset G$ and consider the Cayley (sum) graph $\G^* = X^*(G,S)$ and the MDCGs $\G_S^*= MX^*(G;S,S)$, $\G_e^*=MX^*(G;S,\{e\})$ and $\G_{S\cup e}^*=MX^*(G;S,S \cup \{e\})$. 
	
	\noindent $(a)$ 
	If $Spec(\G^*) \subset \R$, we have the following.
	
	\begin{enumerate}[$(i)$]
		\item If $\G_e^*$ is hypoenergetic, then $\G_S^*$ is hypoenergetic. \sk 
		
		\item If $\G_{S\cup e}^*$ is hypoenergetic, then $\G_S^*$ is hypoenergetic. \sk 
		
		\item If $\G_S^*$ is (weak) hyperenergetic, then $\G_e^*$ and $\G_{S\cup e}^*$ are (weak) hyperenergetic. 
	\end{enumerate} 
	
	\noindent $(b)$ 
	If further, $\G$ is an integral graph, then we have the following. 
	
	\begin{enumerate}[$(i)$]
		\item If $\G_{S\cup e}^*$ is hypoenergetic, then $\G_e^*$ is hypoenergetic. \sk 
		
		\item If $\G_e^*$ is (weak) hyperenergetic, then $\G_{S\cup e}^*$ is (weak) hyperenergetic.
	\end{enumerate}	
\end{prop}

\begin{proof}
	It follows directly from Corollary \ref{coro: desigualdades energia} for the real case, and from Proposition~\ref{prop: integral equienergy} for the integral case.
\end{proof}

Now, we will study these energetic properties for the graphs $\G_S^*$, $\G_e^*$ and $\G_{S\cup e}^*$ separately. For the graphs $\G_e^*$ and $\G_{S\cup e}^*$ we will have to restrict ourselves to the case of real spectrum of $\G^*$.

\subsection{The graphs $\G_S^*$}
We begin by considering the MDCGs in $\mathcal{F}$ of the form $\G_S^*$. 
In this case, almost all energetic properties hold for $\G_S^*$ if and only if it holds for $\G^*$.

\begin{prop} \label{prop: hypo-hyper GS}
Let $G$ be a group, $S\subset G$, and consider the graphs $\G^*=X^*(G,S)$ and $\G_S^*=MX^*(G;S,S)$. Then, we have: 

\noindent $(a)$ $\G_S^*$ is hypoenergetic if and only if $\G^*$ is hypoenergetic. 

\noindent $(b)$ $\G_S^*$ is orderenergetic if and only if $\G^*$ is orderenergetic.  

\noindent $(c)$ If $\G_S^*$ is hyperenergetic, then $\G^*$ is hyperenergetic. Furthermore, if 
$\G^*$ is weak hyperenergetic,
then $\G_S^*$ is hyperenergetic if and only if $\G^*$ is hyperenergetic.
\end{prop}

\begin{proof}
Recall that $\G_S^*$ has $2n$ vertices if $\G^*$ has $n$, with energy 
	$\mathcal{E}(\G^*_S) = 2\mathcal{E}(\G^*)$.

\noindent $(a)$ If the graph $\G_S^*$ is hypoenergetic, we have that
	$\mathcal{E}(\G^*_S) = 2\mathcal{E}(\G^*)<2n$.
Then, $\mathcal{E}(\G^*)<n$ and, hence, $\G^*$ is hypoenergetic.
Conversely, if $\G^*$ is hypoenergetic, we have that $\mathcal{E}(\G^*)<n$. Then,
	$\mathcal{E}(\G^*_S) = 2\mathcal{E}(\G^*)<2n$,
and $\G_S^*$ is hypoenergetic.

\noindent $(b)$ 
Analogous to the proof of ($a$), by changing all the inequalities by equalities. 

\noindent $(c)$ If $\G_S^*$ is hyperenergetic, then 
	$ 2\mathcal{E}(\G^*) = \mathcal{E}(\G_S^*)>4n-2 $.
Hence, $\mathcal{E}(\G^*)>2n-1>2n-2$, and hence the graph $\G^*$ is hyperenergetic.
On the other hand, if $\G^*$ is hyperenergetic and we assume that $\mathcal{E}(\G_S^*)>2n-1$, then we have that
	$ \mathcal{E}(\G_S^*)=2\mathcal{E}(\G^*)>2(2n-1)=4n-2$.
Thus, the graph $\G_S^*$ is hyperenergetic.
\end{proof}

\begin{exam}[\textit{hypoenergetic graphs}] \label{exam: hypoenergetic}
Up to isomorphism, there are exactly ten integral non-trivial hypoenergetic Cayley graphs $\G=X(G,S)$ and ten integral non-trivial hypoenergetic Cayley sum graphs $\G^+=X^+(G,S)$ over groups $G$ of order $|G|\le 15$.
Representatives of the isomorphism classes of these groups are given in the next two tables. 
 
\begin{table}[H]
	\centering
	{\footnotesize
		\begin{tabular}{c c c c}
			\hline
			$G$ & $S$ & $\mathcal{E}(\G)$ & $|G|$ \\
			\hline
			
			$\mathbb{D}_3$ & $\{(0,1),(1,0)\}$ & $4$ & $6$ \\
			
			$\mathbb{D}_4$ & $\{(0,1),(1,0)\}$ & $4$ & $8$ \\
			
			$\mathbb{D}_5$ & $\{(0,1),(1,0),(1,1),(2,0)\}$ & $8$ & $10$ \\
			
			$\mathbb{D}_6$ & $\{(0,1),(1,0)\}$ & $8$ & $12$ \\
			
			$\mathbb{D}_6$ & $\{(0,1),(2,0)\}$ & $8$ & $12$ \\
			
			$\mathbb{D}_6$ & $\{(0,1),(1,0),(3,1),(4,0)\}$ & $8$ & $12$ \\
			
			$\mathbb{A}_4$ & $\{(12)(34), (134), (234)\}$ & $6$ & $12$ \\
			
			$\mathbb{D}_7$ & $\{(0,1),(1,0),(1,1),(2,0),(2,1),(3,0)\}$ & $12$ & $14$ \\
			
			$\mathbb{D}_7$ & $\{(0,1),(1,0),(1,1),(2,0),(3,1),(4,0)\}$ & $12$ & $14$ \\
			
			$\mathbb{D}_7$ & $\{(0,1),(1,0),(1,1),(2,0),(4,0),(5,1)\}$ & $12$ & $14$ \\
			
			\hline
	\end{tabular}}
	\caption{Hypoenergetic non-trivial Cayley graphs.}
	\label{tab:energy_subsets}
\end{table}

\begin{table}[H] 
	\centering
	{\footnotesize
		\begin{tabular}{c c c c}
			\hline
			$G$ & $S$ & $\mathcal{E}(\G)$ & $|G|$ \\
			\hline
			
			$\mathbb{D}_3$ & $\{(0,0),(0,1)\}$ & $4$ & $6$ \\
			
			$\mathbb{D}_3$ & $\{(0,1),(1,0)\}$ & $3$ & $6$ \\
			
			$\mathbb{D}_3$ & $\{(0,0),(0,1),(1,0)\}$ & $5$ & $6$ \\
			
			$\mathbb{D}_3$ & $\{(0,1),(1,0),(1,1)\}$ & $5$ & $6$ \\
			
			$\mathbb{D}_4$ & $\{(0,0),(0,1)\}$ & $6$ & $8$ \\
			
			$\mathbb{D}_4$ & $\{(0,1),(1,0)\}$ & $4$ & $8$ \\
			
			$\mathbb{D}_4$ & $\{(0,1),(2,0)\}$ & $6$ & $8$ \\
			
			$\mathbb{D}_4$ & $\{(0,0),(0,1),(1,0),(1,1)\}$ & $6$ & $8$ \\
			
			$\mathbb{D}_4$ & $\{(0,1),(1,0),(1,1),(2,0)\}$ & $6$ & $8$ \\
			
			$\mathbb{D}_6$ & $\{(0,0),(0,1)\}$ & $8$ & $12$ \\
			
			\hline
	\end{tabular}}
	\caption{Hypoenergetic non-trivial Cayley sum graphs.}
	\label{tab: sum energy_subsets}
	\end{table}
Notice that all the groups involved are of even order and that there are not integral hypoenergetic (sum) Cayley graphs over abelian groups of order at most $15$.

In the tables above, the elements of the groups are denoted as follows. For the dihedral group $\mathbb{D}_n$ of order $2n$, its elements are represented as pairs $(j,i) \in \mathbb{Z}_n \rtimes \mathbb{Z}_2$, where $0 \le j \le n-1$ and $i \in \{0,1\}$. Under the standard presentation $\mathbb{D}_n = \langle r, s \mid r^n = s^2 = 1, srs = r^{-1} \rangle$, the pair $(j,0)$ corresponds to the rotation $r^j$, while $(j,1)$ corresponds to the reflection $r^j s$. 
For the alternating group $\mathbb{A}_4$ of degree 4, we use the standard cycle notation. 
 
For more details about the table, see \cite{ChP3}. By Proposition \ref{prop: hypo-hyper GS}, we have that the graphs $MX(G;S,S)$, with the pair $(G,S)$ as in Table \ref{tab:energy_subsets}, are hypoenergetic graphs; and the graphs $MX^+(G;S,S)$, with the pair $(G,S)$ as in Table \ref{tab: sum energy_subsets}, are hypoenergetic graphs
\hfill $\diamond$
\end{exam}

\begin{rem}[\textit{weak hyperenergetic graphs}] \
	
\noindent ($i$) 
There are $35$ integral weak hyperenergetic Cayley graphs $\G=X(G,S)$, up to isomorphism, over groups $G$ of order $|G|\le 15$, comprising only groups of order $12$ ($\Z_{12}$, $\Z_6\times\Z_2$, $\mathbb{A}_4$) and $\Z_{15}$.

\noindent ($ii$)
Up to isomorphism, there are exactly $66$ integral weak hyperenergetic Cayley sum graphs $\G^+=X^+(G,S)$ with $|G|\le 15$, involving only groups of order $12$ ($\Z_{12}$, $\Z_6\times\Z_2$, $\mathbb{D}_6$, $\mathbb{Q}_{12}$) and $\Z_{15}$.
\end{rem}

\subsection{The graphs $\G_e^*$}
We now focus on the energetic properties of the MDCGs of the form $\G_e^*$. 

\begin{prop}\label{prop: hypo-hyper Ge}
Let $G$ be a group, $S \subset G$ and consider the graphs $\G^*=X^*(G,S)$ and $\G_e^*=MX^*(G;S,\{e\})$. 
Assuming that $Spec(\G^*)\subset\R$ the following holds:

\noindent $(a)$ The graph $\G_e^*$ is not hypoenergetic. \sk

\noindent $(b)$ If $\frak{S}_{\mathcal{J}}^*=\varnothing$, then $\G_e^*$ is orderenergetic if and only if $\G^*$ is orderenergetic. \sk 

\noindent $(c)$ If  $\# \frak{S}_{\mathcal{J}}^*- \mathcal{E}_{\mathcal{J}}(\G^*)\le 1$ and $\G_e^*$ is hyperenergetic, then $\G^*$ is hyperenergetic. On the other hand, if $\# \frak{S}_{\mathcal{J}}^*- \mathcal{E}_{\mathcal{J}}(\G^*)\ge 1$ and $\G^*$ is hyperenergetic, then $\G_e^*$ is hyperenergetic. In particular, if 
    $\# \frak{S}_{\mathcal{J}}^*- \mathcal{E}_{\mathcal{J}}(\G^*) = 1 $
    then $\G_e^*$ is hyperenergetic if and only if $\G^*$ is hyperenergetic.
\end{prop}

\begin{proof}
Recall that the energy of the mirror di-Cayley (sum) graph $\G_e^*$ is given by
	$$ \mathcal{E}(\G^*_e) = 2 \{ \mathcal{E}(\G^*) - \mathcal{E}_\mathcal{J}(\G^*) + \# \frak{S}_\mathcal{J}^* \}. $$

\noindent $(a)$ 
Recall that the eigenvalues of $\G_e^*$ are given by $\lambda^* + 1$ and $\lambda^* - 1$ for each $\lambda^* \in \frak{S}^*$. Thus, the energy of $\G_e^*$ can be expressed as
$$\mathcal{E}(\G^*_e) = \sum_{\lambda^* \in \frak{S}^*} (|\lambda^* + 1| + |\lambda^* - 1|).$$
Applying the triangle inequality, we have that 
$|\lambda^* + 1| + |\lambda^* - 1| \ge |(\lambda^* + 1) - (\lambda^* - 1)| = 2$. Summing this lower bound over all $n$ eigenvalues of $\G^*$ yields
$$\mathcal{E}(\G^*_e) \ge \sum_{\lambda^* \in \frak{S}^*} 2 = 2n.$$
Since the mirror di-Cayley graph $\G_e^*$ has exactly $2n$ vertices, 
$\G_e^*$ can never be hypoenergetic.


\noindent $(b)$ Since $\frak{S}_\mathcal{J}^*=\varnothing$, then 
	$- \mathcal{E}_\mathcal{I}(\G^*) + \# \frak{S}_\mathcal{I}^* =0$.
If the graph $\G_e^*$ is orderenergetic, we have that
	$ 2 \mathcal{E}(\G^*) = \mathcal{E}(\G^*_e) = 2n$,
and the graph $\G^*$ is orderenergetic. On the other hand, if the graph $\G^*$ is orderenergetic, then
	$$ \mathcal{E}(\G^*_e) = 2 \{ \mathcal{E}(\G^*) - \mathcal{E}_\mathcal{J}(\G^*) + \# \frak{S}_\mathcal{J}^* \} = 2 \mathcal{E}(\G^*) = 2n,$$
and the graph $\G^*$ is orderenergetic.

\noindent $(c)$ If the graph $\G^*$ is hyperenergetic and $\# \frak{S}_\mathcal{J}^*- \mathcal{E}_\mathcal{J}(\G^*)\le 1$, then
$$\mathcal{E}(\G_e^*) =  2\{ \mathcal{E}(\G^*) - \mathcal{E}_\mathcal{J}(\G^*) + \# \frak{S}_\mathcal{J}^* \} > 2(2n-2)+2 = 4n-2,$$
and the graph $\G_e^*$ is hyperenergetic. On the other hand, if  $\G_e^*$ is hyperenergetic and it holds that $\# \frak{S}_\mathcal{J}^*- \mathcal{E}_\mathcal{J}(\G^*)\le 1$, then 
$$\mathcal{E}(\G^*)>2n-1 - \# \frak{S}_\mathcal{J}^*- \mathcal{E}_\mathcal{J}(\G^*) \ge 2n-2,$$
and the graph $\G^*$ is hyperenergetic.
\end{proof}

Notice that, in particular, we have that 
	$$ \G_e^* \text{ is order energetic} \qquad \Leftrightarrow \qquad \G^* \text{ is order energetic}$$ 
provided that $\frak{S}_{\mathcal{J}}^*=\varnothing$, and 
that under the assumption $\# \frak{S}_{\mathcal{J}}^*- \mathcal{E}_{\mathcal{J}}(\G^*) = 1$ we have 
	$$ \G_e^* \text{ is hyperenergetic} \qquad \Leftrightarrow \qquad \G^* \text{ is hyperenergetic}. $$ 
However, the conditions 
$\frak{S}_{\mathcal{J}}^*=\varnothing$ and $\# \frak{S}_{\mathcal{J}}^* = \mathcal{E}_{\mathcal{J}}(\G^*) +1 \ge 1$ 
can not hold simultaneously.

\sk 
As a consequence of the previous proposition, we obtain the next result for integral spectrum.

\begin{coro} \label{coro: hypo-hyper Ge integral}
Let $G$ be a group, $S \subset G$ and $\G^*=X^*(G,S)$. 
Assuming that $Spec(\G)\subset \Z$ the following holds: 

\begin{enumerate}[$(a)$]
	\item The graph $\G_e^*$ is not hypoenergetic. \sk 
	
	\item If $m_0^*=0$, then $\G_e^*$ is orderenergetic if and only if $\G^*$ is orderenergetic. \sk 
	
	\item If $\G_e^*$ is hyperenergetic and $m_0^* \le 1$, then $\G^*$ is hyperenergetic. On the other hand, if $\G^*$ is hyperenergetic and $m_0^* \ge 1$, then $\G_e^*$ is hyperenergetic.
\end{enumerate} 
\end{coro}

\begin{proof}
 Assume that $Spec(\Gamma) \subset \mathbb{Z}$. Consequently, the only possible eigenvalue in the open interval $\mathcal{J} = (-1, 1)$ is $0$. 
 
 This implies that $\mathfrak{S}_{\mathcal{J}}^* \subset \{0\}$. Therefore, the number of eigenvalues in this interval is exactly the multiplicity of zero, so $\#\mathfrak{S}_{\mathcal{J}}^* = m_0^*$, and the restricted energy is $\mathcal{E}_{\mathcal{J}}(\Gamma^*) = 0$. In particular, the expression $\#\mathfrak{S}_{\mathcal{J}}^* - \mathcal{E}_{\mathcal{J}}(\Gamma^*)$ simplifies to exactly $m_0^*$, and the condition $\mathfrak{S}_{\mathcal{J}}^* = \varnothing$ is equivalent to $m_0^* = 0$.
 
 Items $(a)$ and $(b)$ follow immediately by substituting $\mathfrak{S}_{\mathcal{J}}^* = \varnothing$ with $m_0^* = 0$ in items $(a)$ and $(b)$ of Proposition \ref{prop: hypo-hyper Ge}.
 
 For item $(c)$, substituting $\#\mathfrak{S}_{\mathcal{J}}^* - \mathcal{E}_{\mathcal{J}}(\Gamma^*) = m_0^*$ into $(c)$ of Proposition \ref{prop: hypo-hyper Ge} yields that if $\Gamma_e^*$ is hyperenergetic and $m_0^* \le 1$, then $\Gamma^*$ is hyperenergetic. Conversely, if $\Gamma^*$ is hyperenergetic and $m_0^* \ge 1$, then $\Gamma_e^*$ is hyperenergetic.	
\end{proof}

In particular, we have 
	\begin{itemize}
		\item If $m_0^*=0$, then $\G_e^*$ is order-energetic \quad $\Leftrightarrow$ \quad $\G^*$ is order-energetic. \msk 
		
		\item If $m_0^*=1$, then $\G_e^*$ is hyper-energetic \quad $\Leftrightarrow$ \quad $\G^*$ is hyper-energetic.
\end{itemize}

\begin{exam} \label{exam: hyper Se}
To illustrate how mirror di-Cayley graphs can become hyperenergetic even when their underlying cover is not (see ($c$) of the corollary),
consider the group $G=\Z_8$ and the 
subset $S=\{1,2,3,5,6,7\}$. The spectrum of the integral Cayley graph $\G = X(G,S)$ is given by
	$$ Spec(\G^*) = \{ [6]^1, [0]^4, [-2]^3 \}. $$
The energy of this graph is $\mathcal{E}(\G) = 1 \cdot 6 + 4 \cdot 0 + 3|-2| = 12$. Since the threshold for hyperenergeticity is $2\cdot 8-2 = 14$, the underlying graph $\G$ is strictly not hyperenergetic.
		
However, notice that the multiplicity of the zero eigenvalue is $m_0 = 4$. Then, the energy of its corresponding mirror di-Cayley graph $\G_e$ is
	$$ \mathcal{E}(\G_e) = 2\mathcal{E}(\G) + 2m_0^* = 2 \cdot 12 + 4\cdot 2 = 32. $$
The mirror di-Cayley graph $\G_e$ has $16$ vertices, making its hyperenergetic threshold $2 \cdot 16-2 = 30$. Since $32 > 30$, $\G_e$ is strictly hyperenergetic. 
\hfill $\diamond$
\end{exam}

\subsection{The graphs $\G_{S\cup e}^*$}
Finally, we consider the MDCGs in $\mathcal{F}$ of the form $\G_{S\cup e}^*$. 

\begin{prop}\label{prop: hypo-hyper GSe}
Let $G$ be a group, $S \subset G$ and consider the graphs $\G^*=X^*(G,S)$ and $\G_{S \cup e}^*=MX^*(G;S,S\cup \{e\})$.
If $Spec(\G^*)\subset\R$, we have the following.
	
\noindent $(a)$ The graph $\G_{S \cup e}^*$ is not hypoenergetic. \sk 
	
\noindent $(b)$ If $\G^*$ is orderenergetic, then $\G_{S \cup e}^*$ is not orderenergetic. \sk
	
\noindent $(c)$ If $\G^*$ is hyperenergetic, then $\G_{S \cup e}^*$ is hyperenergetic.
\end{prop}

\begin{proof}
Recall that the energy of the mirror di-Cayley (sum) graph $\G_{S\cup e}^*$ is given by
	$$ \mathcal{E}(\G^*_{S\cup e}) = 2 \{ \mathcal{E}(\G^*) - 2\mathcal{E}_\mathcal{K}(\G^*) + \# \frak{S}_{\mathcal{K}_\infty}^* \}. $$	

\noindent $(a)$ 
Recall that the eigenvalues of $\G_{S \cup e}^*$ consist of $2\lambda^* + 1$ for each $\lambda^* \in \frak{S}^*$, along with $n$ copies of $-1$. Thus, the energy of $\G_{S \cup e}^*$ is given by
	$$ \mathcal{E}(\G^*_{S \cup e}) = \sum_{\lambda^* \in \frak{S}^*} |2\lambda^* + 1| + \sum_{i=1}^n |-1| = \sum_{\lambda^* \in \frak{S}^*} |2\lambda^* + 1| + n. $$
Applying the generalized triangle inequality, the sum of the absolute values is bounded below by the absolute value of the sum:
	$$ \sum_{\lambda^* \in \frak{S}^*} |2\lambda^* + 1| \ge \Big| \sum_{\lambda^* \in \frak{S}^*} (2\lambda^* + 1) \Big| = \Big| 2 \sum_{\lambda^* \in \frak{S}^*} \lambda^* + \sum_{\lambda^* \in \frak{S}^*} 1 \Big|. $$
Recall from Section $2$ that we assume $e \notin S$ for the graph $\G_{S \cup e}^*$. Consequently, the underlying Cayley cover $\G^*$ has no loops, meaning the sum of its eigenvalues 
is exactly zero; and thus the inequality simplifies to
	$$ \sum_{\lambda^* \in \frak{S}^*} |2\lambda^* + 1| \ge n.$$ 
Substituting this back into the total energy equation yields:
	$$ \mathcal{E}(\G^*_{S \cup e}) \ge n + n = 2n. $$
Since the mirror di-Cayley graph $\G_{S \cup e}^*$ has exactly $2n$ vertices and its energy is bounded below by $2n$, it can never be strictly less than its order. Therefore, $\G_{S \cup e}^*$ is not hypoenergetic.

\noindent $(b)$ 
If the graph $\G^*$ is orderenergetic, then
$ \mathcal{E}(\G^*_{S\cup e}) = 2 \{ n - 2\mathcal{E}_\mathcal{K}(\G^*) + \# \frak{S}_{{\mathcal{K}_\infty}}^* \} $.
We have that
\begin{equation}\label{eq: grado de regularidad}
	- 2\mathcal{E}_\mathcal{K}(\G^*) + \# \frak{S}_{{\mathcal{K}_\infty}}^* \ge \# \frak{S}_{{\mathcal{K}_\infty}}^* - \#\frak{S}_{\mathcal{K}}^* = \# (\frak{S}^* \cap [0,\infty)) \ge 1,
\end{equation}
	since $|S| \in \# (\frak{S}^* \cap [0,\infty))$.
Then,	$ \mathcal{E}(\G^*_{S\cup e}) > 2(n+1) = 2n+2 > 2n $ and the graph $\G_{S \cup e}^*$ is not orderenergetic.

\noindent $(c)$ 
If the graph $\G^*$ is hyperenergetic, we have that
	$$ \mathcal{E}(\G^*_{S\cup e}) > 2 (2n-2) + 2 \{-2\mathcal{E}_\mathcal{K}(\G^*) + \# \frak{S}_{{\mathcal{K}_\infty}}^* \}.$$
By \eqref{eq: grado de regularidad}, we have that
	$ \mathcal{E}(\G^*_{S\cup e}) > 2 (2n-2) +2 =4n-2$,
and the graph $\G_{S \cup e}^*$ is hyperenergetic.
\end{proof}

\begin{exam}
As in Example \ref{exam: hyper Se}, we can observe how the mirror di-Cayley graph $\G_{S \cup e}$ can also achieve hyperenergeticity from a non-hyperenergetic base cover. Let us reuse the integral Cayley graph $\G = X(\Z_8,S)$ with energy $\mathcal{E}(\G) = 12$. The underlying graph $\G$ is strictly not hyperenergetic.
		
For the family $\G_{S \cup e}$, the energy depends on $s_0$, the number of non-negative eigenvalues of $\G$. From its spectrum $Spec(\G^*) = \{ [6]^1, [0]^4, [-2]^3 \}$, we see that $s_0 = 5$.
By Proposition \ref{prop: integral equienergy}, the energy of the corresponding mirror graph $\G_{S \cup e}$ is given by
	$$ \mathcal{E}(\G_{S \cup e}) = 2\mathcal{E}(\G) + 2s_0  = 34. $$
Since $\G_{S \cup e}$ has $16$ vertices, its threshold for hyperenergeticity is $30$. Because $34 > 30$, the graph $\G_{S \cup e}$ is strictly hyperenergetic.
\hfill $\diamond$
\end{exam}

To finish the section we give a general example of hyperenergy for all the families.
\begin{exam}
For a positive integer $n\ge 4$, let $G=\mathbb{S}_n$ be the symmetric group on $n$ letters, and $Z_n$ be the set of all $n$-cycles of $G$. By Theorem 1.1 in \cite{ebrahimi}, the Cayley graph $X(\mathbb{S}_n,Z_n)$ is an integral hyperenergetic graph with
	$$ \mathcal{E}(X(G,Z_n))=2^{n-1}(n-1)!$$ 
Hence, we have the following:

\noindent $(i)$ Since $\mathcal{E}(X(G,Z_n))=2^{n-1}(n-1)!>2n-1$, the graph $MX(\mathbb{S}_n;Z_n,Z_n)$ is hyperenergetic by Proposition \ref{prop: hypo-hyper GS}.

\noindent $(ii)$ By Proposition \ref{prop: hypo-hyper Ge} and the fact that
$$ m_0 (X(\mathbb{S}_n,Z_n))=n!-\tbinom{2n-2}{n-1},$$
we obtain that the graph $MX(\mathbb{S}_n;Z_n,\{e\})$ is hyperenergetic.

\noindent $(iii)$ By Proposition \ref{prop: hypo-hyper GSe}, the graph $MX(\mathbb{S}_n;Z_n,Z_n\cup\{e\})$ is hyperenergetic.
\hfill $\diamond$
\end{exam} 


\section{Crossed equienergy}
\label{sec: equienergy X, X+}
Now, we study crossed equienergy in the family $\mathcal{F}$, that is between non-sum and sum di-Cayley graphs.
More precisely, we give some necessary conditions for $MX(G;S,T)$ and $MX^+(G;S,T)$ with $T\in \mathcal{S} = \{ \{e\}, S, S \cup \{e\} \}$ to be equienergetic and non-isospectral graphs.

\subsection{The case $T=S$}
This case is not so interesting since the situation for the Cayley (sum) graphs pass through directly to the di-Cayley (sum) graphs.

\begin{prop} \label{prop: equienergy C1}
Consider the Cayley (sum) graphs $\G^*=X^*(G,S)$ and the mirror di-Cayley (sum) graphs $\G^*_S=MX^*(G;S,S)$.
Then, $\{\G_S,\G^+_S\}$ are equienergetic and non-isospectral if and only if $\{\G,\G^+\}$ are equienergetic and non-isospectral. \end{prop}

\begin{proof}
This is automatic from \eqref{eq: Ebi=2E} in Theorem \ref{teo: E(MDCG)=2E+cosas}, by using Theorem 6.3 in \cite{ChP}.
\end{proof}

We now illustrate the proposition.
\begin{exam}
By Example 2.8 and Example 2.11 in \cite{PV}, we have that the odd cycles $C_{2n+1}=X(\Z_{2n+1},\{\pm 1\})$ and the odd paths with loops at the ends $\mathring{P}_{2n+1}=X^+(\Z_{2n+1},\{\pm 1\})$ are equienergetic and non-isospectral graphs. Then, by Proposition \ref{prop: equienergy C1}, we have that 
	$$ MX(\Z_{2n+1};\{\pm 1\},\{\pm 1\}) \qquad \text{ and } \qquad MX^+(\Z_{2n+1};\{\pm 1\},\{\pm 1\}) $$ 
are equienergetic and non-isospectral graphs for every $n\in\N$.
\hfill $\diamond$
\end{exam}

\subsection{The case $T=\{e\}$}
This case is more involved.

Let $m_0^*$ denote the multiplicity of $0$ as an eigenvalue of $\G^*=X^*(G,S)$, that is 
\begin{equation} \label{eq: m0s}
	m_0^* = m_{\G^*}(0).
\end{equation}	
We begin by showing that in the case of real spectrum, if $\G^*$ has no eigenvalues in the interval $\mathcal{J}=(-1,1)$ or if $0$ is the only eigenvalue in this interval, then the mirror di-Cayley graphs $\{\G_e,\G^+_e\}$ are equienergetic and non-isospectral if and only if the Cayley graphs $\{\G,\G^+\}$ are equienergetic and non-isospectral, provided the multiplicity of $0$ is the same in each Cayley graph.

\begin{prop} \label{prop: case 2}
Consider the Cayley (sum) graphs $\G^*=X^*(G,S)$ and the mirror di-Cayley (sum) graphs $\G^*_e = MX^*(G;S,\{e\})$. 
If $Spec(\G) \subset \R$, $\frak{S}^*_\mathcal{J} \subset \{0\}$ and $m_0=m_0^+$,
then $\{\G_e,\G^+_e\}$ are equienergetic and non-isospectral if and only if $\{\G,\G^+\}$ are equienergetic and non-isospectral.
\end{prop}

\begin{proof}
Since the spectrum $\G^*$ is real, by \eqref{eq: Ebi=2E+cosas}, we have that $\G_e$ and $\G^+_e$ are equienergetic if and only if
	$$ 2\big\{ \mathcal{E}(\G) - \mathcal{E}_\mathcal{J}(\G) + \# \frak{S}_\mathcal{J} \big\} = 2\big\{ \mathcal{E}(\G^+) - \mathcal{E}_\mathcal{J}(\G^+) + \# \frak{S}_\mathcal{J}^+ \big\}.$$
	
Since $\frak{S}^*_\mathcal{J} \subset \{0\}$, then $\mathcal{E}_\mathcal{J}(\G) = \mathcal{E}_\mathcal{J}(\G^+) = 0$, $\# \frak{S}_\mathcal{J} = m_0$, and $\# \frak{S}_\mathcal{J}^+ = m_0^+$. Hence, since $m_0 = m_0^+$, $\G_e$ and $\G^+_e$ are equienergetic if and only if $\G$ and $\G^+$ are equienergetic. 
	
Furthermore, the non-isospectrality of the graphs is straightforward from Theorem~6.3 in \cite{ChP}.
\end{proof}

The spectra of $X(G,S)$ and $X^+(G,S)$ can be computed by using the irreducible characters $\hat{G}$ of $G$. 

\begin{lem} \label{lem: eigenvalues X*GS}
Let $G$ be a finite group and $S$ a normal subset of $G$. Then, we have:

\begin{enumerate}[$(a)$]
    \item The eigenvalues of the Cayley graph $\Gamma=X(G,S)$ are given by
			$$ {\rm Eig}(\G) = \Big\{ \lambda_\chi(\G) = \tfrac{1}{\chi(1)} \chi(S): \chi \in \hat{G} \Big\}.$$
			
    \item The eigenvalues of the Cayley sum graph $\G^+=X^+(G,S)$ are real and given by
    $$ {\rm Eig}(\G^+) = \{ \lambda_\chi = \tfrac{1}{\chi(1)}\chi(S)\}_{\chi \in \hat{G}_\R} \cup \{ \lambda_\chi^\pm = \pm \tfrac{1}{\chi(1)}\chi(S)\}_{\chi \notin \hat{G}_\R} \subset \R, $$ 
    with $\hat{G}_{\R}$ the subgroup of real characters of $G$.
\end{enumerate}
In both cases, $\chi(S)$ is as defined in \eqref{eq: chiS}.
\end{lem}

\begin{proof}
	For ($a$), see Corollary $3.2$ in \cite{Babai}, for $G$ abelian, and Theorem $2.1$ in \cite{DVGM}, for the general case. For ($b$), see Theorem $1$ in \cite{Z}, for $G$ abelian, and Proposition $4.4$ in \cite{ChP4}, for the general case.
\end{proof}

Next, we will show that the multiplicity of the eigenvalue $0$ is the same in the graphs $X(G,S)$ and $X^+(G,S)$, provided that $S$ is a normal subset.

We will need first the following generalization of Lemma $2.3$ in \cite{PV} for $G$ non-abelian.

\begin{lem} \label{lem: autovalores conjugados}
Let $G$ be a finite group and $S$ a normal subset of $G$, with $e \notin S$. Then 
$\lambda_{\overline{\chi}} = \overline{\lambda_\chi}$.
In particular, $\lambda_\chi = 0$ if and only if $\lambda_{\overline{\chi}} = 0$. 
\end{lem} 

\begin{proof}
By Lemma~\ref{lem: eigenvalues X*GS}~$(a)$, the eigenvalue associated with $\chi$ is $\lambda_\chi(\G) = \frac{\chi(S)}{\chi(1)}$. For the conjugate character $\overline{\chi} \in \hat{G}$, we have $\overline{\chi}(s) = \overline{\chi(s)}$ for all $s \in G$ and $\overline{\chi}(1) = \chi(1)$. Therefore,
	$$ \lambda_{\overline{\chi}}(\G) = \tfrac{1}{\overline{\chi}(1)} \overline{\chi}(S) = \tfrac{1}{\chi(1)}\sum_{s \in S} \overline{\chi(s)} = \overline{ \tfrac{1}{\chi(1)}\sum_{s \in S} \chi(s) } = \overline{\lambda_\chi(\G)}. $$
Consequently, 
we conclude that $\lambda_\chi(\G) = 0$ if and only if $\lambda_{\overline{\chi}}(\G) = 0$.
\end{proof}

We are now in a position to show that the multiplicity of the 0 eigenvalue is the same in normal Cayley and Cayley sum graphs associated to the same pair $(G,S)$.

\begin{lem} \label{lema: m0=m0+} 
Let $G$ be a group and $S$ a normal subset of $G$, such that $e \notin S$. 
Then, in the notation of \eqref{eq: m0s}, we have that
	$$m_0=m_0^+.$$
\end{lem}

\begin{proof} 
Suppose that $\chi\in\hat{G}$ with $\chi(S)=0$. If $\overline{\chi}=\chi$, that is $\chi\in\hat{G}_\R$, the contribution of $\chi$ to the multiplicity of $0$ is $\chi(1)^2$  to both $\G$ and $\G^+$, since the eigenvalue associated to $\chi$ is the same as the eigenvalue associated to $\overline{\chi}$.

On the other hand, if $\overline{\chi}\neq\chi$, then $\overline{\chi}(S)=0$ by Lemma \ref{lem: autovalores conjugados}. In this case, in $\G$, the pair $\{\chi, \overline{\chi}\}$ contributes $0$ with total multiplicity $2(\chi(1))^2$; and in $\G^+$, the eigenvalues $\lambda_\chi^\pm = \pm \frac{1}{\chi(1)}|\chi(S)| = 0$ yield the exact same total multiplicity $2(\chi(1))^2$.
Summing over all irreducible characters with $\chi(S) = 0$, we conclude that $m_0 = m_0^+$.
\end{proof}

Now, we give a simple criterion ensuring that $\G_e$ and $\G^+_e$ are equienergetic and non-isospectral graphs when $S$ is a normal subset of $G$. 
Namely, we show that if in Proposition \ref{prop: case 2} we consider $S$ normal, we do not need the assumption on the multiplicities of the zero eigenvalue. 
Notice that, in general, checking that $S$ is normal is easier than checking $m_0=m_0^+$.

\begin{prop}  \label{prop: equienerg no-iso C2} 
Let $G$ be a group and $S \subset G$ normal. Put $\G_e^*=MX^*(G;S,\{e\})$ and $\G^*=X^*(G,S)$. 
If $Spec(\G)\subset \R$ and $\frak{S}^*_\mathcal{J} \subset \{0\}$, then 
$\{\G_e,\G_e^+\}$ 
are equienergetic non-isospectral graphs if and only if $\{\G,\G^+\}$  
are equienergetic non-isospectral graphs. 
\end{prop}

\begin{proof}
Let $\mathcal{J}=(-1,1)$. 
If $\frak{S}^* \cap \mathcal{J} = \varnothing$, then $\# \frak{S}^*_\mathcal{J} =0$ and 
	$\sum_{\lambda^*\in\frak{S}^*_\mathcal{J}} |\lambda^*|=0.$
Therefore, by Theorem \ref{teo: E(MDCG)=2E+cosas},
	$ \mathcal{E}(\G_e) = 2\mathcal{E}(\G)$ and 
	$\mathcal{E}(\G^+_e) = 2\mathcal{E}(\G^+)$.
Hence, $\mathcal{E}(\G_e)=\mathcal{E}(\G^+_e)$ if and only if $\mathcal{E}(\G)=\mathcal{E}(\G^+)$.
	
On the other hand, if $\frak{S}^*_\mathcal{J} = \{\!\{0\}\!\}$, then we have that
$\#\frak{S}^*_\mathcal{J} = m_0^*$ and $\sum_{\lambda^*\in \frak{S}^*_\mathcal{J}}|\lambda^*|=0$.
Therefore, by Theorem \ref{teo: E(MDCG)=2E+cosas},
$\mathcal{E}(\G_e) = 2\mathcal{E}(\G) + m_0$ and $\mathcal{E}(\G^+_e) = 2\mathcal{E}(\G^+) + m_0^+$.
By Lemma \ref{lema: m0=m0+}, we have that $\mathcal{E}(\G_e) = \mathcal{E}(\G^+_e)$ if and only if $\mathcal{E}(\G)=\mathcal{E}(\G^+)$.
	
In both cases, the non-isospectrality of the graphs is immediate from Theorem~6.3 in \cite{ChP}.
\end{proof}

If the Cayley graph is integral, we do not need the assumption $\frak{S}^*_\mathcal{J} \subset \{0\}$ in the previous proposition.
\begin{coro} \label{coro: integral equienergy C2}
Let $G$ be a finite 
group and $S$ be a normal subset of $G$. Put $\G_e^*=MX^*(G;S,\{e\})$ and $\G^*=X^*(G,S)$. If $\G$ is integral then 
$\{\G_e,\G_e^+\}$ are equienergetic non-isospectral if and only if $\{\G,\G^+\}$ 
are equienergetic non-isospectral. 
\end{coro}

\begin{proof}
Since $X(G,S)$ is integral then $X^+(G,S)$ is integral and, hence $\frak{S}^*\cap\mathcal{I}\subset\{0\}$. Then, by Proposition \ref{prop: equienerg no-iso C2}, $MX(G;S,\{e\})$ and $MX^+(G;S,\{e\})$ are equienergetic non-isospectral graphs if and only if $X(G,S)$ and $X^+(G,S)$ are equienergetic non-isospectral graphs.
\end{proof}

We now give sufficient conditions for obtaining equienergetic non-isospectral pairs of the form $\{\G_S, \G_S^+\}$ and $\{\G_e, \G_e^+\}$.

\begin{coro} \label{coro: G_S,G_S+ equienergetic}
Let $G$ be an abelian group and $S$ be a symmetric subset of $G$ such that the trivial character $\chi_0$ is the only real character of $G$. If $G$ has a non-trivial character $\chi$ such that $-\lambda_{\chi}$ is not an eigenvalue of $\G$, then $\{\G_S, \G_S^+\}$ are equienergetic non-isospectral graphs. Moreover, if $\frak{S}^*_\mathcal{J} \subset \{0\}$, then  $\{\G_e, \G_e^+\}$ are equienergetic non-isospectral graphs.
\end{coro}
	
\begin{proof}
By Theorem $2.7$ and Proposition $2.10$ in \cite{PV} we have that the graphs $\G$ and $\G^+$ are equienergetic non-isospectral graphs. Then, by Proposition \ref{prop: equienergy C1}, the graphs $\G_S$ and $\G_S^+$ are equienergetic non-isospectral graphs. If $\frak{S}^*_\mathcal{J} \subset \{0\}$, by Proposition \ref{prop: equienerg no-iso C2}, the graphs $\G_e$ and $\G_e^+$ are equienergetic and non-isospectral.
\end{proof}

We now illustrate the previous corollary.
\begin{exam}
Consider the unitary Cayley graphs $X^*(\Z_n,\Z_n^*)$.
As shown in Example 2.8 in \cite{PV}, for $n = p^k$ where $p$ is an odd prime, the Cayley graph $\G = X(\Z_n, \Z_n^*)$ and the Cayley sum graph $\G^+ = X^+(\Z_n, \Z_n^*)$ are equienergetic but non-isospectral (and the conditions on the characters in Corollary \ref{coro: G_S,G_S+ equienergetic} hold). 
Since these graphs are integral, it follows from Corollary \ref{coro: G_S,G_S+ equienergetic} that their corresponding mirror extensions 
	$$ \{\G_e,\G_e^+\} \qquad \text{and} \qquad \{\G_S,\G_S^+\} $$
are also equienergetic and non-isospectral.
\hfill $\diamond$
\end{exam}

\subsection{The case $T=S\cup\{e\}$}
In this case we do not have to assume that $S$ is normal or that the spectrum is real, but there is a difference with Case 1 ($T=S$) in that we need extra assumptions to ensure that the equienergy of non-isospectral pairs of MCDGs pass through the Cayley covers and conversely.

\begin{prop} \label{prop: equienergy X,X+ case 3}
Let $G$ be group and $S \subset G$. 
Put $\G_{S \cup e}^* = MX^*(G;S,\G\cup e)$ and $\G^*=X^*(G,S)$. 
If 
	$ \frak{S}^*_\mathcal{K} = \varnothing$ and 
	$s_0=s_0^+$,
where  $s_{0}^{*} = \# \frak{S}^*_{\R_{\ge 0}}$, 
then $\{\G_{S\cup e}, \G_{S\cup e}^+\}$ are equienergetic and non-isospectral if and only if $\{\G,\G^+\}$ are equienergetic and non-isospectral.
\end{prop}

\begin{proof}
If $\frak{S}^*\cap [-\tfrac{1}{2},0)=\varnothing$, then 
	$\sum_{\lambda^*\in\frak{S}^*_\mathcal{J}} |\lambda^*|=0$,
and $\frak{S}^*\cap [-\tfrac{1}{2},\infty) = \frak{S}^*\cap [0,\infty)=s_0^*$. Thus, if $s_0=s_0^+$, then by Proposition \ref{teo: E(MDCG)=2E+cosas}, we have that
	$ \mathcal{E}( \G_{S\cup e}) = 2 \mathcal{E}(\G^*)$,
so $\G_{S\cup e}$ and $\G^+_{S\cup e}$ are equienergetic graphs if and only if $\G$ and $\G^+$ 
are equienergetic graphs.
	
The non-isospectrality of the graphs involved follows directly from Proposition~6.3 in \cite{ChP}. 
\end{proof}

\subsubsection*{Integral spectrum} 
As in the previous case, if the Cayley graph is integral, we do not need to assume that both the Cayley graph and the Cayley sum graph have no eigenvalues in the interval $[-\tfrac{1}{2},0)$.

\begin{coro} \label{coro: equienergy X, X+ Sue}
Let $G$ be a group and $S \subset G$. 
Put $\G_{S \cup e}^*=MX^*(G;S,\G\cup e)$ and $\G^*=X^*(G,S)$. 
If $\G$ is integral and $s_0=s_0^+$ 
then $\{\G_{S\cup e}, \G_{S\cup e}^+\}$ are equienergetic and non-isospectral if and only if $\{\G,\G^+\}$ are equienergetic and non-isospectral.
\end{coro}

\begin{proof}
Since $X(G,S)$ is integral, then $X^+(G,S)$ is integral and, hence $\frak{S}^*_\mathcal{K} = \varnothing$. Therefore, since $s_0=s_0^+$, by Proposition \ref{prop: equienergy X,X+ case 3}, $\G_{S\cup e}$ and $\G_{S\cup e}^+$ are equienergetic non-isospectral graphs if and only if $\G$ and $\G^+$ are equienergetic non-isospectral graphs.
\end{proof}

\begin{exam}
To illustrate why the condition $s_0 = s_0^+$ is strictly necessary in Corollary \ref{coro: equienergy X, X+ Sue}, we can find a pair of integral non-equienergetic graphs $\G$ and $\G^+$ where $s_0 \neq s_0^+$ while the mirror di-Cayley graphs $\G_{S\cup e}$ and $\G_{S\cup e}^+$ are equienergetic non-isospectral graphs.

Let $G = \mathbb{D}_3$ and $S = \{(0,1), (1,0)\}$. The spectrum of $\G$ and $\G^+$ are given by
	$$Spec(\G) = \{[2]^1, [0]^3, [-1]^2\} \qquad \text{and} \qquad Spec(\G^+) = \{[2]^1, [0]^4, [-1]^1\}. $$
We can easily see that their energies are $\mathcal{E}(\G) = 4$ and $\mathcal{E}(\G^+) = 3$. Hence, $\G$ and $\G^+$ are clearly not equienergetic. 

On the other hand, counting the non-negative eigenvalues yields $s_0 = 4$ and $s_0^+ = 5$. By Proposition \ref{prop: integral equienergy}, we have that
	$$\mathcal{E}(\G_{S\cup e}) = 2\mathcal{E}(\G) + 2s_0 = 8+8 = 16 \quad \text{and} \quad \mathcal{E}(\G_{S \cup e}^+) = 2\mathcal{E}(\G^+) + 2s_0^+ = 10+6 = 16. $$
Therefore, the mirror di-Cayley graphs $\G_{S \cup e}$ and $\G_{S \cup e}^+$ are equienergetic graphs.
\hfill $\diamond$
\end{exam}

\section{Equienergy} \label{sec: equienergy}
In this final section, we study equienergy between pairs of MDCGs in two different situations. In first place, we consider pairs of graphs having the same covers but different diconnection sets, that is $MX^*(G;S,T)$ and $MX^*(G;S,T')$. In second place, we deal with the most general case of pairs of graphs with different covers, that is $MX^*(G_1;S_1,T_1)$ and $MX^*(G_2;S_2,T_2)$, with $(G_1,S_1) \ne (G_2,S_2)$, although $T_1$ and $T_2$ of the same type, i.e.\@ either $T_i=S_i$, or $T_i=\{e_i\}$ or else $T_i=S_i\cup \{e_i\}$ for $i=1,2$.

\subsection{Mixed equienergy (same covers)} \label{subsec: mixed equienergy}
Here we consider the equienergy problem for MDCGs in the family $\mathcal{F}$ having the same covers, but different di-connection sets. 
That is, we seek for equienergy between 
the graphs 
$$ \G_T^* = MX^*(G;S,T) \qquad \text{and} \qquad \G_{T'}^* = MX^*(G;S,T') $$ 
for $T,T'\in \mathcal{S}=\{\{e\},S,S\cup\{e\}\}$ with $T\ne T'$. 
In this case we will have to assume that the spectrum of the underlying Cayley graph is real.

In the previous notations we have the following result for equienergy.

\begin{prop} \label{prop: equinergy TyT'}
Let $G$ be a group, $S\subset G$ and $\G=X(G,S)$. 
If $Spec(\G) \subset \R$, then $\G_e^*$ and $\G_S^*$ are non-isospectral equienergetic if and only if 
$\frak{S}_\mathcal{J}^* = \varnothing$. 
\end{prop}

\begin{proof} 
The graphs $\G^*_e$ and $\G^*_S$ are equienergetic if and only if
	$ \mathcal{E}(\G^*_e) = \mathcal{E}(\G^*_S) $.
By Proposition~\ref{teo: E(MDCG)=2E+cosas}, this is true if and only if 
	\begin{equation}\label{eq: equienergy case 1}
		2 \mathcal{E}(\G^*) + 2\#\frak{S}_\mathcal{J}^* - 2\mathcal{E}_\mathcal{J}(\G^*) = 2\mathcal{E}(\G^*),
	\end{equation} 
which is the same as
	$\#\frak{S}_\mathcal{J}^* = \mathcal{E}_\mathcal{J}(\G^*)$.
		
Now, if $\frak{S}_\mathcal{J}^* = \varnothing$, we have that 
	$\#\frak{S}_\mathcal{J}^*=0$ 
	and $\mathcal{E}_\mathcal{J}(\G^*)=0$,
and thus $\G^*_e$ and $\G^*_S$ are equienergetic.
		
If $\frak{S}_\mathcal{J}^* \ne \varnothing$, for $\lambda^* \in (-1,1)$ we have that $|\lambda^*|<1$. Therefore, 
	$$ \mathcal{E}_\mathcal{J}(\G^*) = \sum_{\lambda^*\in\frak{S}^*_\mathcal{J}} |\lambda^*| < \sum_{\lambda^*\in\frak{S}^*_\mathcal{J}} 1 = \# \frak{S}^*_\mathcal{J}. $$ 
Hence, these graphs are not equienergetic.
To see that these graphs are non-isospectral, we refer to Proposition $6.1$ in \cite{ChP}.
\end{proof}

\begin{rem}
In Theorem 8 in \cite{Bonifacio}, Bonifacio, Vinagre and Abreu proved that if $\G$ is a connected graph with $|\lambda| \ge 1$ for all the eigenvalues $\lambda$ of $\G$ --that is, $\mathfrak{S}_{\mathcal{J}} = \varnothing$ in our notation--, then 
	$$ \{ \G \Box P_2, \G \times P_2\} \quad \text{are equienergetic non-isospectral} $$ 
graphs (they used the alternative notations $\otimes$ for $\Box$ and $K_2$ for $P_2$).  

In Proposition \ref{prop: equinergy TyT'} we have obtained a similar result for $\G$ any Cayley (sum) graph $\G=X^*(G,S)$, without requiring connectivity, but instead asking for real spectrum. There, we showed that 
	$$ \{ \G \Box P_2, \G \times \mathring{P}_2\} \quad \text{are equienergetic non-isospectral} $$ 
graphs. In this way, both results complement each other.
\end{rem}

We now show, that in the case of real spectrum of the covers, the graphs $\G_S^*$ and $\G^*_{S\cup e}$ are not equienergetic.

\begin{prop} \label{prop: no equinergy TyT'}
Let $G$ be a group, $S\subset G$ and $\G=X(G,S)$. If $Spec(\G) \subset \R$, then $\G_S^*$ and $\G^*_{S\cup e}$ are not equienergetic.
\end{prop}

\begin{proof}
Notice that if $\G^*$ in integral, then the result follows from the inequalities in Proposition~\ref{prop: integral equienergy}. For the non-integral case, by Proposition \ref{teo: E(MDCG)=2E+cosas},
the graphs $\G^*_S$ and $\G^*_{S\cup e}$ are equienergetic if and only if 
	$$ 2\mathcal{E}(\G^*) = 2\mathcal{E}(\G^*) + 2\#\frak{S}_{{\mathcal{K}_\infty}}^* -4\mathcal{E}_{\mathcal{K}}(\G^*),$$
where $\mathcal{K}=[-\tfrac{1}{2},0)$ and ${\mathcal{K}_\infty}=\mathcal{K}\cup\R_{\ge 0}$,
which is the same as
	$$ \#\frak{S}_{{\mathcal{K}_\infty}}^* = 2\mathcal{E}_{\mathcal{K}}(\G^*).$$
	
For $\lambda^*\in[-\tfrac{1}{2},0]$ we have that $|\lambda^*|\le\tfrac{1}{2}$, and hence
	$$ 
	2\mathcal{E}_{\mathcal{K}}(\G^*) = 2 \sum_{\lambda^*\in\frak{S}^*\cap(-\frac{1}{2},0]}|\lambda^*|\le 2\sum_{\lambda^*\in\frak{S}^*\cap(-\frac{1}{2},0]} \tfrac{1}{2} = \#\frak{S}_{\mathcal{K}}^*< \#\frak{S}_{{\mathcal{K}_\infty}}^*,
	$$
so these graphs are not equienergetic.
\end{proof}	

Next, again in the case of real spectrum of the covers $\G$, we give some sufficient conditions on $\G$ ensuring that $\G_e^*$ and $\G^*_{S\cup e}$ are not equienergetic.

\begin{prop}
Let $G$ be a group, $S\subset G$ and $\G=X(G,S)$. If $Spec(\G) \subset \R$, we have the following:
\begin{enumerate}[$(a)$]
	\item If $\G^*$ is integral, then $\G_e^*$ and $\G^*_{S\cup e}$ are not equienergetic. \sk
	
	\item If $\G^*$ has at most one eigenvalue in $\mathcal{J}$, then $\G_e^*$ and $\G^*_{S\cup e}$ are not equienergetic. \sk 
	
	\item If $\frak{S}_{\mathcal{J}}^*$ is symmetric, then $\G_e^*$ and $\G^*_{S\cup e}$ are not equienergetic. 
\end{enumerate}
\end{prop}

\begin{proof}
\noindent $(a)$ Since $\G$ is integral, by the inequalities in Proposition \ref{prop: integral equienergy}, we have that $\G_e$ and $\G_{S\cup e}$ are not equienergetic. \sk

\noindent $(b)$ By Proposition \ref{teo: E(MDCG)=2E+cosas}, 
the graphs $\G^*_e$ and $\G^*_{S\cup e}$ are equienergetic if and only if
		$$2\mathcal{E}(X^*(G,S))+2\#\frak{S}_{\mathcal{J}}^*-2\mathcal{E}_{\mathcal{J}}(\G^*) = 2\mathcal{E}(X^*(G,S)) + 2\#\frak{S}_{{\mathcal{K}_\infty}}^* 		-4\mathcal{E}_{\mathcal{K}}(\G^*),$$
which is the same as
\begin{equation} \label{eq: equienergy case c}
	\#\frak{S}_{\mathcal{J}}^*-\mathcal{E}_{\mathcal{J}}(\G^*)= \#\frak{S}_{{\mathcal{K}_\infty}}^* -2\mathcal{E}_{\mathcal{K}}(\G^*).
\end{equation}
		
Notice that $\#\frak{S}_{{\mathcal{K}_\infty}}^*\in\Z$ and $\#\frak{S}_{{\mathcal{K}_\infty}}^*=m_{|S|}^*+\#\frak{S}_{{\mathcal{K}_\infty} \smallsetminus\{|S|\}}^*$, since $|S|\in {\mathcal{K}_\infty}$ is the degree of regularity of $X^*(G,S)$.  

If $\frak{S}_{\mathcal{J}}^*=\varnothing$, then $\#\frak{S}_{\mathcal{J}}^*=0$, $\mathcal{E}_{\mathcal{J}}(\G^*)=0$ and $\mathcal{E}_{\mathcal{K}}(\G^*)=0$. Therefore, $\#\frak{S}_{{\mathcal{K}_\infty}}^*=0$. But $\#\frak{S}_{{\mathcal{K}_\infty}}^*\ge m_{|S|}^* > 0$, which is a contradiction. Hence, The graphs $\G_e$ and $\G_{S\cup e}$ are not equienergetic. 
		
If $\frak{S}_{\mathcal{J}}^*=\{\lambda_0^*\}$ for some $\lambda_{0}^*$, we have that 
		$$	\#\frak{S}_{\mathcal{J}}^*=m_{\lambda_0}^*,\qquad\text{and}\qquad \mathcal{E}_{\mathcal{J}}(\G^*)=m_{\lambda_0}^*|\lambda_0^*|.$$
Now, we need to consider three distinct cases. 
		
\noindent $(i)$ 
If $\lambda_0^*\in(-1,-\tfrac{1}{2})\cup(0,1)$, then
	$$ m_{\lambda_0}^* - m_{\lambda_0}^*|\lambda_0^*| = \#\frak{S}_{\mathcal{J}}^* - \mathcal{E}_{\mathcal{J}}(\G^*) = \#\frak{S}_{{\mathcal{K}_\infty}}^*  -2\mathcal{E}_{\mathcal{K}}(\G^*) = \#\frak{S}_{{\mathcal{K}_\infty}}^*,$$
which is a contradiction since $m_{\lambda_0}^*+ m_{\lambda_0}^*|\lambda_0^*|\in \R \smallsetminus \Q$ and $\#\frak{S}_{{\mathcal{K}_\infty}}^*\in\Z$.
		
		\sk
		
\noindent $(ii)$ 
If $\lambda_0^* \in [-\tfrac{1}{2},0)$, then \eqref{eq: equienergy case c} is equivalent to
	$$m_{\lambda_0}^*- m_{\lambda_0}^*|\lambda_0^*|= \#\frak{S}_{\mathcal{J}}^*- \mathcal{E}_{\mathcal{J}}(\G^*)= \#\frak{S}_{{\mathcal{K}_\infty}}^*-2\mathcal{E}_{\mathcal{K}}(\G^*)= \#\frak{S}_{{\mathcal{K}_\infty}}^*-2m_{\lambda_0}^*|\lambda_0^*|,$$
that is
	$$m_{\lambda_0}^* + m_{\lambda_0}^*|\lambda_0^*|=\#\frak{S}_{{\mathcal{K}_\infty}}^*.$$
This is a contradiction since $\#\frak{S}_{{\mathcal{K}_\infty}}^*\in\Z$ and $m_{\lambda_0}^* + m_{\lambda_0}^*|\lambda_0^*| \in \R \smallsetminus \Q$.
		
		\sk
		
\noindent $(iii)$ 
If $\lambda_0^*=0$, by \eqref{eq: equienergy case c}, we have that 
	$ \#\frak{S}_{{\mathcal{K}_\infty}} = 0 $,
which is impossible because $\#\frak{S}_{{\mathcal{K}_\infty}}\ge m_{|S|}^*>0$.
		
Thus, $\G_e$ and $\G_{S\cup e}$ are not equienergetic. \sk

\sk 

\noindent $(c)$ 
Assume that $\frak{S}_\mathcal{J}^*$ is symmetric. By \eqref{eq: equienergy case c}, the graphs $\G_e^*$ and $\G_{S \cup e}^*$ are equienergetic if and only if
	$$ \#\frak{S}_\mathcal{J}^* - \mathcal{E}_\mathcal{J}(\G^*) = \#\frak{S}_{\mathcal{K}_\infty}^* - 2\mathcal{E}_{\mathcal{K}}(\G^*). $$

The restricted energies simplify to $\mathcal{E}_\mathcal{J}(\G^*) = 2 \sum_{\lambda^* \in \frak{S}^* \cap (0, 1)} \lambda^*$ and $\mathcal{E}_\mathcal{K}(\G^*) =  \sum_{\lambda^* \in \frak{S}^* \cap (0, \frac 12]} \lambda^*$, by symmetry.
Substituting these into the equation and canceling the common sum over the interval $(0, \frac 12]$ yields:
	$$ 2 \sum_{\lambda^* \in \frak{S}^* \cap (\frac 12, 1)} \lambda^* = \#\frak{S}_\mathcal{J}^* - \#\frak{S}_{\mathcal{K}_\infty}^* \in \Z. $$
On the left-hand side, since $\G^*$ is a graph, its eigenvalues are algebraic integers. A fundamental property is that any rational algebraic integer must be a regular integer. Since the eigenvalues in our sum are strictly confined to the fractional interval $(\frac 12, 1)$, they cannot be integers and are therefore irrational numbers. The restricted sum to $(\frac 12, 1)$ of irrational algebraic integers cannot equal an integer (we leave the details), yielding a contradiction. Thus, $\G_e^*$ and $\G_{S \cup e}^*$ are not equienergetic.
\end{proof}

\subsection{Equienergy (different covers)} 
\label{subsec: general equienergy}
Here, we consider two different group-subset pairs $(G_1,S_1)$ and $(G_2,S_2)$ and give conditions for the mirror di-Cayley graphs given by these pairs to be equienergetic.

We begin with graphs having $T_i=S_i$ for $i=1,2$.	

\begin{prop} \label{prop: equienergy dif covers C1}
Let $G_1$, $G_2$ be two groups, $S_1$ be a subset of $G_1$ and $S_2$ be a subset of $G_2$. If $Spec(X(G_1,S_1)), Spec(X(G_2,S_2))\subset \R$, we have that the graphs $\G_{S_1}^*=MX^*(G_1;S_1,S_1)$ and $\G_{S_2}^*=MX^*(G_2;S_2,S_2)$ are equienergetic if and only if the graphs $X^*(G_1,S_1)$ and $X^*(G_2,S_2)$ are equienergetic.
\end{prop}
	
\begin{proof}
This follows directly from Theorem \ref{teo: E(MDCG)=2E+cosas}.	
\end{proof}
 
\begin{exam}
Consider $G_1 = G_2 = \Z_{12}$, $S_1 = \{4,8\}$ and $S_2 = \{2,4,8,10\}$. The spectra of the Cayley graphs $\G_1 = X(G_1,S_1)$ and $\G_2 = X(G_2,S_2)$ are given by
	$$ Spec(\G_1) = \{[4]^2, [0]^6, [-2]^4\} \qquad \text{and} \qquad Spec(\G_2) = \{[2]^4, [-1]^8\}. $$
Both graphs have integral spectra and have energy equal to $\mathcal{E}(\G_1)=\mathcal{E}(\G_2)=16$.
Since $\G_1$ and $\G_2$ are equienergetic, Proposition \ref{prop: equienergy dif covers C1} guarantees that their corresponding mirror di-Cayley graphs with $T_i = S_i$ are also equienergetic. Indeed, the energies of $\G_{1,S_1}$ and $\G_{2,S_2}$ are given by
	$$ \mathcal{E}(\G_{1, S_1}) = 2\mathcal{E}(\G_1) = 32 \qquad \text{and} \qquad \mathcal{E}(\G_{2, S_2}) = 2\mathcal{E}(\G_2) = 32, $$
which confirms the result.
\hfill $\diamond$ 
\end{exam}

Now, we consider MDCGs having $T_i=\{e_i\}$ for $i=1,2$.		
	
\begin{prop} \label{prop: equienergy dif covers C2}
Let $G_1$, $G_2$ be groups, $S_1 \subset G_1$, $S_2 \subset G_2$ and put $\G_i^*=X^*(G_i,S_i)$ and 
$\G_{i,e_i}^* = MX^*(G_i;S_i,\{e_i\})$ for $i=1,2$. If $\frak{S}_1^*\cap (-1,1) = \frak{S}_2^*\cap (-1,1)$ as a multiset, then 
$\G_{1,e_1}^*$ and $\G_{2,e_2}^*$ are equienergetic if and only if $\G_1^*$ and $\G_2^*$ are equienergetic.
\end{prop}

\begin{proof}
By Theorem \ref{teo: E(MDCG)=2E+cosas}, the energy of the mirror di-Cayley graphs $\Gamma_{i,e_i}^*$ for $i=1,2$ are given by
	$$ \mathcal{E}(\Gamma_{i,e_i}^*) = 2\{\mathcal{E}(\Gamma_i^*) - \mathcal{E}_\mathcal{J}(\Gamma_i^*) + \#\mathfrak{S}_{i,\mathcal{J}}^*\}, $$
where $\mathcal{J} = (-1, 1)$. 
	
By hypothesis, the multisets $\mathfrak{S}_{1}^* \cap \mathcal{J}$ and $\mathfrak{S}_{2}^* \cap \mathcal{J}$ are equal. This equality implies that both the number of eigenvalues in this interval and their restricted energies (the sum of their absolute values) are perfectly identical for both covers. That is,
	$$ \#\mathfrak{S}_{1,\mathcal{J}}^* = \#\mathfrak{S}_{2,\mathcal{J}}^* 
		\qquad \text{and} \qquad 
		\mathcal{E}_\mathcal{J}(\Gamma_1^*) = \mathcal{E}_\mathcal{J}(\Gamma_2^*). $$
	
Let $C = \#\mathfrak{S}_{1,\mathcal{J}}^* - \mathcal{E}_\mathcal{J}(\Gamma_1^*) = \#\mathfrak{S}_{2,\mathcal{J}}^* - \mathcal{E}_\mathcal{J}(\Gamma_2^*)$. Substituting this constant back into the initial energy equations, we obtain:
	$$ \mathcal{E}(\Gamma_{1,e_1}^*) = 2\mathcal{E}(\Gamma_1^*) + 2C 
		\qquad \text{and} \qquad 
		\mathcal{E}(\Gamma_{2,e_2}^*) = 2\mathcal{E}(\Gamma_2^*) + 2C. $$
	
Therefore, $\mathcal{E}(\Gamma_{1,e_1}^*) = \mathcal{E}(\Gamma_{2,e_2}^*)$ if and only if $\mathcal{E}(\Gamma_1^*) = \mathcal{E}(\Gamma_2^*)$; i.e.\@ $\G_{1,e_1}^*$ and $\G_{2,e_2}^*$ are equienergetic if and only if $\G_1^*$ and $\G_2^*$ are equienergetic. 
\end{proof}

\begin{exam}
Consider $G_1=G_2=G_3=\Z_{12}$ and the subsets $S_1=\{0,4,6,8\}$, $S_2=\{2,3,9,10\}$ and $S_3=\{3,4,8,9\}$. 
The spectra of the graphs $\G_i=X(G_i,S_i)$, for $i=1,2,3$ are given by
\begin{align*}
 Spec(\G_1) &= \{[4]^2,[2]^2,[1]^4,[-1]^4\},\\
 Spec(\G_2) &= \{[4]^1,[1]^6,[0]^1,[-2]^2,[-3]^2\},\\
 Spec(\G_3) &= \{[4]^1,[2]^2,[1]^2,[0]^1,[-1]^4,[-3]^2\},
\end{align*}
and their energies coincide $\mathcal{E}(\G_1) = \mathcal{E}(\G_2) = \mathcal{E}(\G_3) = 20$.

Notice that $\frak{S}_1^*\cap (-1,1) = \varnothing \ne \{0\} = \frak{S}_2^*\cap (-1,1) = \frak{S}_3^*\cap (-1,1)$ and the energies of the corresponding di-Cayley graphs are 
	$$ \mathcal{E}(\G_{1,e_1}) =  40, \qquad \text{and} \qquad \mathcal{E}(\G_{2,e_2}) = \mathcal{E}(\G_{3,e_3}) = 42.$$
This shows that since $\Gamma_2$ and $\Gamma_3$ are equienergetic ($\mathcal{E} = 20$) and share the exact same restricted multiset $\mathfrak{S}^* \cap (-1, 1) = \{0\}$, their corresponding mirror di-Cayley graphs $\Gamma_{2,e_2}$ and $\Gamma_{3,e_3}$ are also equienergetic with an energy of $42$. Conversely, although $\Gamma_1$ shares the same initial energy of $20$, it fails the multiset condition since $\varnothing \neq \{0\}$. As a direct result, its mirror extension $\Gamma_{1,e_1}$ yields an energy of $40$ and is not equienergetic to the others. This demonstrates that having equal energy on the underlying covers is insufficient on its own, and the condition matching the eigenvalues in the interval $(-1, 1)$ is strictly necessary.
\hfill $\diamond$
\end{exam}

Finally, we show that under mild conditions of the spectra of the covers, the graphs $\G_{1,S_1 \cup e_1}^*$ and $\G_{2,S_2 \cup e_2}^*$ are equienergetic if and only if $\G_1^*$ and $\G_2^*$ are equienergetic graphs.

\begin{prop} \label{prop: equienergy dif covers C3}
Let $G_1$, $G_2$ be groups, $S_1 \subset G_1$, $S_2 \subset G_2$ and put $\G_i^*=X^*(G_i,S_i)$ and $\G_{i, S_i \cup e_i}^*= MX^*(G_i;S_i,S_i\cup \{e_i\})$ for $i=1,2$. 
If $\frak{S}_1^* \cap [-\tfrac{1}{2},0) = \frak{S}_2^* \cap [-\tfrac{1}{2},0) = \varnothing$ and $s_0(\G_1^*) = s_0(\G_2^*)$, then 
$\G_{1,S_1 \cup e_1}^*$ and $\G_{2,S_2 \cup e_2}^*$ are equienergetic if and only if $\G_1^*$ and $\G_2^*$ are equienergetic.
\end{prop}

\begin{proof}
Notice that if $\frak{S}_1^*\cap[-\tfrac{1}{2},0) = \frak{S}_2^*\cap [-\tfrac{1}{2},0) =\varnothing$, then 
	$$ \mathcal{E}(\G_{i, S_i \cup e_i}^*) = 2 \mathcal{E}(\G_i^*) + s_0(\G_i^*), $$
for $i=1,2$.
Since $s_0(\G_1^*) = s_0(\G_2^*)$, we have that $\G_{1,S_1 \cup e_1}^*$ and $\G_{2,S_2 \cup e_2}^*$ are equienergetic if and only if $\G_1^*$ and $\G_2^*$ are equienergetic.
\end{proof}

We now show that, in general, $\G_{1,S_1 \cup e_1}^*$ and $\G_{2,S_2 \cup e_2}^*$ can be equienergetic although $\G_{1}^*$ and $\G_{2}^*$ are not.
\begin{exam}
Let $G_1= \Z_{12}$, $G_2 = \mathbb{D}_6$ and consider the subsets $S_1 = \{4,8\} \subset G_1$ and $S_2 = \{(0,1),(1,0),(2,1)\} \subset G_2$ (in the notation of Example \ref{exam: hypoenergetic}). 
We have that
	$$ Spec(\G_1) = \{[2]^4,[-1]^8\} \qquad \text{and} \qquad Spec(\G_2) = \{[3]^1,[1]^3,[0]^4,[-1]^3,[-3]^1\}.$$
Then, $s_0(\G_1) = 4 \ne 8= s_0(\G_2)$ and $\mathcal{E}(\G_1) = 16 \ne 12=\mathcal{E}(\G_2)$, but 
	$$\mathcal{E}(\G_{1,S_1\cup e_1}) = \mathcal{E}(\G_{2,S_2\cup e_2}) = 40.$$
This shows that condition $s_0(\G_1) = s_0(\G_2)$ is a necessary condition.
\hfill $\diamond$
\end{exam}

\end{document}